\documentclass[letterpaper,12pt]{amsart}
\usepackage{amsmath}
\usepackage{amsthm}
\usepackage{amssymb}
\usepackage{amsfonts}

\usepackage{latexsym,array,delarray,amsthm,amssymb,epsfig,setspace,color, tikz}%eufrak

\usetikzlibrary{plotmarks}

\newtheorem{thm}{Theorem}[section]
\newtheorem{lemma}[thm]{Lemma}
\newtheorem{prop}[thm]{Proposition}
\newtheorem{cor}[thm]{Corollary}
\newtheorem{conj}[thm]{Conjecture}
\newtheorem*{thm*}{Theorem}
\newtheorem*{lemma*}{Lemma}
\newtheorem*{prop*}{Proposition}
\newtheorem*{cor*}{Corollary}
\newtheorem*{conj*}{Conjecture}

\theoremstyle{definition}
\newtheorem{defn}[thm]{Definition}
\newtheorem{ex}[thm]{Example}
\newtheorem*{ex*}{Example}

\newtheorem{notation}[thm]{Notation}

\theoremstyle{remark}
\newtheorem{rmk}[thm]{Remark}

\newcommand{\Z}{\mathbb{Z}}
\newcommand{\N}{\mathbb{N}}
\renewcommand{\P}{\mathbb{P}}

\newcommand{\R}{\mathbb{R}}
\newcommand{\C}{\mathbb{C}}

\newcommand{\bfa}{\mathbf{a}}

\newcommand{\bfc}{\mathbf{c}}

\newcommand{\bfe}{\mathbf{e}}

\newcommand{\bfh}{\mathbf{h}}

\newcommand{\Ical}{\mathcal{I}}
\newcommand{\Mcal}{\mathcal{M}}

\newcommand{\Hcal}{\mathcal{H}}

\newcommand{\conv}{\mathrm{conv}}
\newcommand{\ind}{\mbox{$\perp \kern-5.5pt \perp$}}

\newcommand{\Max}{\mathrm{Max}}
\newcommand{\Int}{\mathrm{Int}}
\newcommand{\rows}{\mathrm{rows}^{j_0}}
\newcommand{\rowsone}{\mathrm{rows}^{j_1}}
\newcommand{\cols}{\mathrm{cols}}
\newcommand{\Abar}{{\overline{A}}}
\newcommand{\ubar}{\overline{u}}
\newcommand{\Psibar}{\overline{\Psi}}
\newcommand{\cbar}{\overline{\bfc}}
\newcommand{\Bbar}{\overline{B}}
\newcommand{\rowspan}{\mathrm{rowspan}}

\title[Quasi-independence models with rational MLE]{Quasi-independence models with rational maximum likelihood estimator}

\author{Jane Ivy Coons}
\author{Seth Sullivant}
\begin{document}

%\begin{frontmatter}

\maketitle

\begin{abstract}
We classify the two-way quasi-independence models (independence models with structural zeros)
that have
rational maximum likelihood estimators, or MLEs.
We give a necessary and sufficient condition on the bipartite graph associated
to the model for the MLE to be rational.
In this case, we give an explicit formula for the MLE
in terms of combinatorial features of this graph.
We also use the Horn uniformization to show that 
for general log-linear models $\Mcal$ with rational MLE,
any model obtained by restricting to a face of the cone of sufficient statistics
of $\Mcal$ also has rational MLE.
\end{abstract}

%\begin{keyword}
%Algebraic statistics, maximum likelihood estimation, quasi-independence models, log-linear models, maximum likelihood degree
%\end{keyword}
%\end{frontmatter}

\section{Introduction}\label{sec:IndepModelsWithSZ}

Huh \cite{huh2014} classified the varieties with rational maximum likelihood estimator
using Kapranov's Horn uniformization \cite{kapranov1991}.  In spite of the classification,
it can be difficult to tell a priori whether a given model has rational MLE, or not.
Duarte, Marigliano, and Sturmfels \cite{duarte2019} have since applied Huh's
ideas to varieties that are the closure of discrete statistical models.
In the present paper, we study this problem for a family of discrete statistical models
called quasi-independence models, also commonly known as independence models with structural zeros.
Because quasi-independence models have a simple structure whose description is determined by a 
bipartite graph, this is a natural test case for trying to apply Huh's theory.
Our complete classification of quasi-independence models with rational MLE
is the main result of the present paper (Theorems \ref{thm:Intro} and \ref{thm:Main}).

Let $X$ and $Y$ be two discrete random variables with $m$ and $n$ states, respectively.
Quasi-independence models describe the situation in which
 some combinations of states of $X$ and $Y$ cannot occur together,
 but $X$ and $Y$ are otherwise independent of one another.
This condition is known as quasi-independence in the statistics literature \cite{bishop2007}.
Quasi-independence models are basic models that arise 
in data analysis with log-linear models.
For example, quasi-independence models arise
in the biomedical field as rater agreement models \cite{agresti1992, rapallo2005}
and in engineering to model system failures at nuclear plants \cite{colombo1988}.
There is a great deal of literature regarding hypothesis testing under the assumption 
of quasi-independence, see, for example, \cite{bocci2019, goodman1994, smith1995}.
Results about existence and uniqueness of the maximum likelihood estimate
in quasi-independence models as well as 
explicit computations in some cases can be found in \cite[Chapter~5]{bishop2007}.

In order to define quasi-independence models, let $S \subset [m] \times [n]$ 
be a set of indices, where $[m] = \{1,2, \ldots, m\}$.  
These correspond to a matrix with structural zeros 
whose observed entries are given by the indices in $S$.  
We often use $S$ to refer to both the set of indices and the matrix 
representation of this set and abbreviate the
ordered pairs $(i,j)$ in $S$ by $ij$. For all $r$, we denote by $\Delta_{r-1}$ the open $(r-1)$-dimensional probability simplex in $\R^r$,
\[
\Delta_{r-1} := \{ x \in \R^r \mid x_i > 0 \text{ for all } i \text{ and } \sum_{i=1}^r x_i = 1 \}.
\]

\begin{defn}
Let $S \subset [m] \times [n]$.  Index the coordinates of $\R^{m+n}$ by $ (s_1, \dots, s_m, t_1, \dots, t_n) = (s,t)$.   Let $\R^S$ denote the real vector space of dimension $\#S$ whose coordinates
are indexed by $S$.
Define the monomial map $\phi^S: \R^{m + n} \rightarrow \R^{S}$ by
\[
\phi^S_{ij}(s,t) = s_i t_j.
\]
The \emph{quasi-independence model} associated to $S$ is the model,
\[
\Mcal_S := \phi^S(\R^{m+n}) \cap \Delta_{\#S -1}.
\]
\end{defn}

We note that the Zariski closure of $\Mcal_S$ is a toric variety since it is parametrized by monomials.
To any quasi-independence model, we can associate a bipartite graph in the following way.

\begin{defn}
The \emph{bipartite graph associated to $S$}, denoted $G_S$, is the bipartite graph with independent sets $[m]$ and $[n]$ with an edge between $i$ and $j$ if and only if $(i,j) \in S$. The graph $G_S$ is \emph{chordal bipartite} if every cycle of length greater than or equal to 6 has a chord. The graph $G_S$ is \emph{doubly chordal bipartite} if every cycle of length greater than or equal 6 has at least two chords. We say that $S$ is doubly chordal bipartite if $G_S$ is doubly chordal bipartite. 
\end{defn}

Let $u \in \N^{S}$ be a vector of counts of independent, identically distributed (iid) data.
The \emph{maximum likelihood estimate}, or MLE, for $u$ in $\Mcal_S$ is the
distribution $\hat{p} \in \Mcal_S$ that maximizes the probability of observing
the data $u$ over all distributions in the model.
We describe the maximum likelihood estimation problem in more detail in Section \ref{sec:LogLinModels}.
We say that $\Mcal_S$ has \emph{rational MLE} if for generic choices of $u$,
the MLE for $u$ in $\Mcal_S$ can be written as a rational function in the entries of $u$.
We can now state the key result of this paper.

\begin{thm}\label{thm:Intro}
Let $S \subset [m] \times [n]$ and let $\Mcal_S$ be the associated quasi-independence model.
Let $G_S$ be the bipartite graph associated to $S$.
Then $\Mcal_S$ has rational maximum likelihood estimate if and only if
$G_S$ is doubly chordal bipartite.
\end{thm}

Theorem \ref{thm:Main} is a strengthened version of Theorem \ref{thm:Intro}
in which we give an explicit formula for the MLE when $G_S$ is doubly chordal bipartite.
The outline of the rest of the paper is as follows. 
In Section \ref{sec:LogLinModels}, we introduce general log-linear models
and their MLEs and discuss some key results on these topics.
In Section \ref{sec:FacialSubsets}, we discuss the notion of a facial  submodel of a log-linear model
and prove that facial submodels of models with rational MLE also have rational MLE.
In Section \ref{sec:FacialSZ}, we apply the results of Section \ref{sec:FacialSubsets}
to show that if $G_S$ is not doubly chordal bipartite, then $\Mcal_S$ does not have rational MLE.
The main bulk of the paper is in Sections \ref{sec:Cliques}, \ref{sec:FixedColumn} and \ref{sec:BirchsThm},
where we show that if $G_S$ is doubly chordal bipartite, then the MLE is rational
and we give an explicit formula for it. Section \ref{sec:Cliques} covers combinatorial features
of doubly chordal bipartite graphs and gives the statement of the main Theorem \ref{thm:Main}.
Sections \ref{sec:FixedColumn} and \ref{sec:BirchsThm} are concerned with the verification
that the formula for the MLE is correct.

%%%%%%%%%%%%%%%%%%%%%%%%%%%%%%%%%%%%%%%%%%%%%%%%%%%%%%%%%%
%%%%%%%%%%%%%%%%%%%%%%%%%%%%%%%%%%%%%%%%%%%%%%%%%%%%%%%%%%
%%%%%%%%%%%%%%%%%%%%%%%%%%%%%%%%%%%%%%%%%%%%%%%%%%%%%%%%%%
%%%%%%%%%%%%%%%%%%%%%%%%%%%%%%%%%%%%%%%%%%%%%%%%%%%%%%%%%%

\section{Log-Linear Models and their Maximum Likelihood Estimates}\label{sec:LogLinModels}

In this section, we collect some results from the literature on log-linear models
and maximum likelihood estimation in these models.  These
results will be important tools in the proof of Theorem \ref{thm:Main}.

Let $A \in \Z^{d \times r}$ with entries $a_{ij}$. 
Denote by $\mathbf{1}$ the vector of all ones in $\Z^r$. 
We assume throughout that $\mathbf{1} \in \rowspan(A)$.

\begin{defn}
The \emph{log-linear model} associated to $A$ is the set of probability distributions,
\[
\Mcal_A := \{ p \in \Delta_{r-1} \mid \log p \in \rowspan(A)\}.
\]
\end{defn}

Algebraic and combinatorial tools are well-suited for the study of log-linear models
since these models have monomial parametrizations. 
Define the map $\phi^A: \R^d \rightarrow \R^r$ by 
\[
\phi^A_j (t_1, \dots, t_d) = \prod_{i=1}^d t_i^{a_{ij}}.
\]
Then we have that $\Mcal_A = \phi^A(\R^d) \cap \Delta_{r-1}$.
Background on log-linear models can be found in \cite[Chapter~6.2]{sullivant2018}.
Denote by $\C[p] := \C[p_1, \dots, p_r]$ the polynomial ring in $r$ indeterminates.
Let $I_A \subset \C[p]$ denote the vanishing ideal of $\phi^A(\R^d)$
over the algebraically closed field $\C$.
Since $\phi^A$ is a monomial map, $I_A$ is a toric ideal.
For this reason, $\Mcal_A$ is also known as a \emph{toric model}.
Some key properties of $I_A$ are summarized in the following proposition.

\begin{prop}[\cite{sullivant2018}, Proposition 6.2.4]
The toric ideal $I_A$ is a binomial ideal and
\[
I_A = \langle p^u - p^v \mid u,v \in \N^r \text{ and } Au = Av \rangle.
\]
If $\mathbf{1} \in \rowspan(A)$, then $I_A$ is homogeneous.
\end{prop}

Note that the quasi-independence model associated to a set 
$S \subset [m] \times [n]$ is a log-linear model
with respect to matrix $A(S)$ constructed in the following way.
We have $A(S) \in \Z^{(m+n) \times \#S}$. The $ij$ column of $A(S)$, denoted $a^{ij}$,
has $k$th entry:
\[
a^{ij}_k = \begin{cases}
1, \text{ if } k = i \\
1, \text{ if } k = m + j \\
0, \text{ otherwise.}
\end{cases}
\]
In this way, $\Mcal_S = \Mcal_{A(S)}$. 
Note that $\mathbf{1} \in \rowspan(A(S))$ for all $S$, since it can be written as the sum of the first $m$ rows of $A(S)$.

Given independent, identically distributed (iid) data $u \in \N^r$,
we wish to infer the distribution $p \in \Mcal_A$ that is ``most likely" to have generated it.
This is the central problem of maximum likelihood estimation.

\begin{defn}
Let $\Mcal$ be a discrete statistical model in $\R^r$ and let $u \in \N^r$ be an iid vector of counts.
The \emph{likelihood function} is
\[
L(p \mid u) = \prod_{i=1}^r p_i^{u_i}.
\]

The \emph{maximum likelihood estimate}, or MLE, for $u$ is the distribution in $\Mcal$ that maximizes the likelihood function; that is,
it is the distribution
\[
\hat{p} =\underset{p \in \Mcal}{\mathrm{argmax}} \ L(p \mid u).
\]
\end{defn}

Note that for a fixed $p \in \Mcal$, $L(p \mid u)$ is exactly the probability of observing
 $u$ from the distribution $p$. 
Hence, the MLE for $u$ is the distribution $\hat{p} \in \Mcal$ that
maximizes the probability of observing $u$.
The map $u \mapsto \hat{p}$ 
is a function of the data known as the \emph{maximum likelihood estimator}.
We are particularly interested in the case when the 
coordinate functions of the maximum likelihood estimator are rational functions of the data.
In this case, we say that $\Mcal$ has \emph{rational MLE}.

The \emph{log-likelihood function} $\ell(p \mid u)$ is the natural logarithm of $L(p \mid u)$.
Note that since the natural log is a concave function, $\ell(p \mid u)$ and $L(p \mid u)$ have the same maximizers.
We define the \emph{maximum likelihood degree} of $\Mcal$ to be
the number of critical points of $\ell(p \mid u)$ for generic $u$.
Huh and Sturmfels \cite{hs2014} show that the maximum likelihood degree is well-defined.
In particular, $\Mcal$ has maximum likelihood degree 1 if and only if it has
rational maximum likelihood estimator \cite{huh2014}.
The following result of Huh gives a characterization of the form of this maximum likelihood estimator, when it exists.

\begin{thm}[\cite{huh2014}]\label{thm:Huh}
A discrete statistical model $\Mcal$ has maximum likelihood degree 1 if and only if
there exists $h = (h_1, \dots, h_r) \in (\C^*)^r$, a positive integer $d$,
and a matrix $B \in \Z^{d \times r}$ with entries $b_{ij}$ whose column sums are zero
such that the map
\[
\Psi: \P^{r-1} \dashrightarrow (\C^*)^r
\]
with coordinate function
\[
\Psi_k(u_1, \dots, u_r) = h_k \prod_{i=1}^d \big(\sum_{j=1}^r b_{ij} u_j \big)^{b_{ik}}
\]
maps dominantly onto $\overline{\Mcal}$. In this case, the function $\Psi$ is
the maximum likelihood estimator for $\Mcal$.
\end{thm}

In this context, the pair $(B,h)$ is called the \emph{Horn pair} that defines $\Psi$, and $\Psi$ is called the \emph{Horn map}. For more details about the Horn map and its connection to the theory of $A$-discriminants, we refer the reader to \cite{duarte2019} and \cite{huh2014}.

\begin{ex}\label{ex:HornMap}
Consider the quasi-independence model 
associated to 
\[
S = \{ (1,1), (1,2), (1,3), (2,1), (2,2), (2,3), (3,1), (3,2) \}.
\]
This is the log-linear model whose defining matrix is
\[
A = \begin{bmatrix}
1 & 1 & 1 & 0 & 0 & 0 & 0 & 0 \\
0 & 0 & 0 & 1 & 1 & 1 & 0 & 0 \\
0 & 0 & 0 & 0 & 0 & 0 & 1 & 1 \\
1 & 0 & 0 & 1 & 0 & 0 & 1 & 0 \\
0 & 1 & 0 & 0 & 1 &0 & 0 & 1 \\
0 & 0 & 1 & 0 & 0 & 1 & 0 & 0
\end{bmatrix}.
\]
We index the columns of $A$ by the ordered pairs in $S$
in the given order.
Note that we have $\Mcal_S = \Mcal_{A(S)}.$
Let $u \in \N^{S}$ be a vector of counts of iid data for the model $\Mcal_S$.

According to Theorem \ref{thm:Intro}, $\Mcal_S$ has rational MLE.
Theorem \ref{thm:Main} shows that the associated Horn pair is
\[
B = \begin{bmatrix}
1 & 1 & 1 & 0 & 0 & 0 & 0 & 0 \\
0 & 0 & 0 & 1 & 1 & 1 & 0 & 0 \\
0 & 0 & 0 & 0 & 0 & 0 & 1 & 1 \\
1 & 0 & 0 & 1 & 0 & 0 & 1 & 0 \\
0 & 1 & 0 & 0 & 1 &0 & 0 & 1 \\
0 & 0 & 1 & 0 & 0 & 1 & 0 & 0 \\
1 & 1 & 0 & 1 & 1 & 0 & 0 & 0 \\
-1 & -1 & -1 & -1 & -1 & -1 & 0 & 0 \\
-1 & -1 & 0 & -1 & -1 & 0 & -1 & -1 \\
-1 & -1 & -1 & -1 & -1 & -1 & -1 & -1 \\
\end{bmatrix}
\]
with $h = (-1, -1, 1, -1, -1, 1, 1, 1)$. The columns of $B$ and $h$ are also indexed by the elements of $S$.
We can use this Horn pair to write the MLE as a rational function of the data. 
Denote by $u_{++}$ the sum of all entries of $u$, and abbreviate each ordered pair $(i,j) \in S$ by $ij$.
Then for example, the $(1,3)$ coordinate of the MLE is
\begin{align*}
\hat{p}_{13} &= h_{13}(u_{11} + u_{12} + u_{13})^1(u_{13} + u_{23})^1(u_{11} + u_{12} + u_{13} + u_{21} + u_{22} + u_{23})^{-1} u_{++}^{-1} \\
&= \frac{(u_{11} + u_{12} + u_{13})(u_{13} + u_{23})}{u_{++} (u_{11} + u_{12} + u_{13} + u_{21} + u_{22} + u_{23})}.
\end{align*}
Similarly, the $(2,3)$ coordinate is
\[
\hat{p}_{23} = \frac{(u_{21} + u_{22} + u_{23})(u_{13} + u_{23})}{u_{++} (u_{11} + u_{12} + u_{13} + u_{21} + u_{22} + u_{23})}.
\]

\end{ex}

The following theorem, known as Birch's Theorem, says that the maximum likelihood estimate for  $u$ in a log-linear model $\Mcal_A$, if it exists, is the unique distribution $\hat{p}$ in $\Mcal_A$ with the same sufficient statistics as the normalized data. A proof of this result can be found in \cite[Chapter~7]{sullivant2018}.

\begin{thm}[Birch's Theorem]\label{thm:Birch}
Let $A \in \Z^{n \times r}$ such that $\mathbf{1} \in \mathrm{rowspan}(A)$. 
Let $u \in \R_{\geq 0}^r$ and let $u_+ = u_1 + \dots + u_r$. 
Then the maximum likelihood estimate in the log-linear model $\Mcal_A$ given data $u$
is the unique solution, if it exists, to the equations $Au = u_+Ap$ subject to $p \in \Mcal_A$.
\end{thm}

\begin{ex*}[Example \ref{ex:HornMap}, continued]
Consider the last row $a_6$ of the matrix $A$. 
One sufficient statistic of $\Mcal_A$ is $a_6 \cdot u = u_{13} + u_{23}$.
We must check that $a_6 \cdot u = u_{++} a_6 \cdot \hat{p}$.
Indeed, we compute that
\begin{align*}
a_6 \cdot \hat{p} &= \frac{(u_{11} + u_{12} + u_{13})(u_{13} + u_{23})}{u_{++} (u_{11} + u_{12} + u_{13} + u_{21} + u_{22} + u_{23})} +
  \frac{(u_{21} + u_{22} + u_{23})(u_{13} + u_{23})}{u_{++} (u_{11} + u_{12} + u_{13} + u_{21} + u_{22} + u_{23})} \\
  &= (u_{13} + u_{23}) \frac{(u_{11} + u_{12} + u_{13} + u_{21} + u_{22} + u_{23})}{u_{++} (u_{11} + u_{12} + u_{13} + u_{21} + u_{22} + u_{23}) } \\
  &= \frac{u_{13} + u_{23}}{u_{++}},
\end{align*}
as needed.

\end{ex*}

%%%%%%%%%%%%%%%%%%%%%%%%%%%%%%%%%%%%%%%%%%%%%%%%%%%%%%%%%
%%%%%%%%%%%%%%%%%%%%%%%%%%%%%%%%%%%%%%%%%%%%%%%%%%%%%%%%%
%%%%%%%%%%%%%%%%%%%%%%%%%%%%%%%%%%%%%%%%%%%%%%%%%%%%%%%%%
%%%%%%%%%%%%%%%%%%%%%%%%%%%%%%%%%%%%%%%%%%%%%%%%%%%%%%%%%

\section{Facial Submodels of Log-Linear Models}\label{sec:FacialSubsets}

In order to prove that a quasi-independence model  with rational MLE
must have a doubly chordal bipartite associated graph $G_S$, we first prove a result
that applies to general log-linear models with rational MLE.
Let $A \in \Z^{n \times r}$ be the matrix defining the monomial map for the log-linear model $\Mcal_A$.
Let $I_A$ denote the vanishing ideal of the Zariski closure of $\Mcal_A$. 
We assume throughout that $\mathbf{1} \in \rowspan(A)$.
Let $P_A = \conv(A)$, where $\conv(A)$ denotes the convex hull of the columns $\bfa_1, \dots, \bfa_r$ of $A$. 

We assume throughout that $P_A$ has $n$ facets, $F_1, \dots, F_n$, 
and that the $ij$ entry of $A$, denoted $a_{ij}$ is equal to the lattice distance 
between the $j$th column of $A$ and facet $F_i$.
This is not a restriction, since one can always reparametrize a log-linear model in this way \cite[Theorem~27]{rauh2011}.
Indeed, given a polytope $Q$, a matrix $A$ that satisfies the above condition
is a \emph{slack matrix} of $Q$, and the convex hull of the columns of $A$ is
affinely isomorphic to $Q$ \cite{gouveia2013}.

Let $\Abar$ be a matrix whose columns are a subset of $A$. 
Without loss of generality, assume that the columns of $\Abar$ are $\bfa_1, \dots, \bfa_s$.

\begin{defn}
The submatrix $\Abar$ is called a \emph{facial submatrix} of $A$ if $P_{\bar{A}}$ is a face of $P_A$.  The corresponding statistical model $\Mcal_\Abar$ is called a \emph{facial submodel}
of $\Mcal_A$.\footnote{Note that the term ``facial submodel'' is a slight abuse of terminology because $\Mcal_\Abar$
is not a submodel of $\Mcal_A$.  This is because the log-linear model $\Mcal_A$ does
not include distributions on the boundary of the probability simplex.  Technically,
$\Mcal_\Abar$ is a submodel of the closure of $\Mcal_A$.}
\end{defn}

Let $\bfe_i$ denote the $i$th standard basis vector in $\R^n$. 
Then $\bfe_i \cdot \bfa_j = 0$ if $\bfa_j$ lies on $F_i$ and 
$\bfe_i \cdot \bfa_j \geq 1$ otherwise. 
So under our assumptions on $A$, this definition of a facial submatrix of $A$ 
aligns with the one given in \cite{geiger2006} and \cite{rauh2011}.
We prove the following result 
concerning the maximum likelihood estimator for $\Mcal_{\Abar}$ 
when $\bar{A}$ is a facial submatrix of $A$.
This result was used implicitly in the proof of Theorem 4.4 of \cite{geiger2006}.

\begin{thm}\label{thm:FacialSubsetMLE}
Let $A \in \Z^{n \times r}$ and let $\Abar \in \Z^{n \times s}$ consist of the first $s$ columns of $A$. 
Suppose that $\Abar$ is a facial submatrix of $A$. 
Let $\Mcal_A$ have rational maximum likelihood estimator $\Psi$ given by the Horn pair $(B,h)$ where $B \in \Z^{d \times r}$
and $h \in (\C^*)^r$. 
Let $\overline{B}$ denote the submatrix consisting of the first $s$ columns of $B$ and let $\bar{h} = (h_1, \dots, h_s)$.
Then $\Mcal_{\Abar}$ has rational maximum likelihood estimator $\Psibar$ given by the Horn pair $(\overline{B},\bar{h})$.
\end{thm}

In order to prove Theorem \ref{thm:FacialSubsetMLE}, we check the conditions of Birch's theorem. 
We do this using the following Lemmas.

\begin{lemma}\label{lem:InVariety}
Let $\Psibar$ be as in Theorem \ref{thm:FacialSubsetMLE}. 
Then for generic $\ubar \in \R_{\geq 0}^s$, $\Psibar(\ubar)$ is defined.
In this case,  $\Psibar(\ubar)$ is in the
Zariski closure of $\Mcal_{\Abar}$.
\end{lemma}

\begin{proof}
Let $u \in \R^r_{\geq 0}$ be given by $u_i = \ubar_i$ if $i \leq s$ and $u_i = 0$ if $i > s$.
We claim that when $\Psi(u)$ is defined,
$\Psibar_k(\ubar) = \Psi_k(u)$ for $k \leq s$.
Indeed, each factor of $\Psi_k(u)$ is of the form
\[
\big( \sum_{j =1}^r b_{ij} u_j \big)^{b_{ik}}
\]
for each $i =1, \dots, d$. If the $i$th factor of $\Psi_k$ is not
identically equal to one, then $b_{ik} \neq 0$.
So the $i$th factor has the nonzero summand $b_{ik}u_k$ and is generically nonzero
when evaluated at a point $u$ of the given form.
In particular, this implies that $\Psi_k(u)$ is defined for a generic $u$
of the given form since having $u_j = 0$ for $j > s$ does not make any factor of $\Psi_k$ identically equal to zero.
Setting each $b_{ij} = 0$ when $j > s$ gives that $\Psibar_k(\ubar) = \Psi_k(u)$ when $k \leq s$.

The elements of $I_{\Abar}$ are those elements of $I_A$ that belong to the polynomial ring $k[p_1, \dots, p_s]$. Let $f \in I_{\Abar}$.
Since $f \in I_A$ as well, $f(\Psibar(\ubar)) = f(\Psi(u)) = 0$, as needed.
\end{proof}

Next we check that the sufficient statistics $\Abar \ubar / \ubar_+$ are equal to those of $\Psibar(\ubar)$.

\begin{lemma}\label{lem:SuffStats}
Let $\cbar$ be a row of $\Abar$. Then
\[
\frac{\cbar \cdot \ubar}{\ubar_+} = \cbar \cdot \Psibar(\ubar).
\]
\end{lemma}

\begin{proof}
Let $\bfc$ be the row of $A$ corresponding to $\cbar$.
Define a sequence $u^{(i)} \in \R^r_{\geq 0}$ by
\[
u_j^{(i)} = \begin{cases}
\ubar_j & \text{ if } j \leq s \\
\epsilon^{(i)}_j & \text{ if } j > s,
\end{cases}
\]
where $\lim_{i \rightarrow \infty} \epsilon^{(i)}_j = 0$ for each $j$.
We choose each $\epsilon^{(i)}_j > 0$ generically so that $\Psi(u^{(i)})$ is defined for all $i$.

Since $\ubar$ is generic, we have that $\lim_{i \rightarrow \infty} u_+^{(i)} = \ubar_+ \neq 0$.
Similarly, we have that $\lim_{i\rightarrow\infty} \bfc \cdot u^{(i)} = \cbar \cdot \ubar$.
So 
\[
\lim_{i\rightarrow\infty} \frac{\bfc \cdot u^{(i)}}{u^{(i)}_+} = \frac{\cbar \cdot \ubar}{\ubar_+}.
\]

Since $\Psi(u^{(i)})$ is the maximum likelihood estimate in $\Mcal_A$ for  each $u^{(i)}$,
by Birch's theorem we have that
\begin{align*}
\frac{\bfc \cdot u^{(i)}}{u^{(i)}_+} &= \bfc \cdot \Psi(u^{(i)}) \\
&= \sum_{i=1}^s c_j \Psi_j(u^{(i)}) +  \sum_{j = s+1}^r c_j \Psi_j(u^{(i)}).
\end{align*}
By the arguments in the proof of Lemma \ref{lem:InVariety}, when $k \leq s$,
no factor of $\Psi_k(u^{(i)})$ involves only summands $u_j^{(i)}$ for $j > s$.
So $\lim_{i\rightarrow\infty}\Psi_k(u^{(i)}) = \Psibar_k(\ubar)$.

Finally, we claim that for $k > s$, $\lim_{i \rightarrow \infty} \Psi_k(u^{(i)}) = 0$.
Without loss of generality, we may assume that $P_{\Abar}$ is a facet of $P_A$.
Indeed, if it were not, we could simply iterate these arguments over a
saturated chain of faces between $P_{\Abar}$ and $P_A$ in the face lattice of $P_A$.
Let $\boldsymbol\alpha = (a_1, \dots, a_r)$ be the row of $A$ corresponding to the facet $P_{\Abar}$ of $P_A$.
Then $a_j= 0$ if $j \leq s$ and $a_j \geq 1$ if $j > s$.
Since $\Psi(u^{(i)})$ is the maximum likelihood estimate in $\Mcal_A$ for $u^{(i)}$,
by Birch's theorem we have that
\begin{align*}
\boldsymbol\alpha \cdot \Psi(u^{(i)}) &= \frac{1}{u^{(i)}_+} (a_{s+1} u^{(i)}_{s+1} + \dots + a_r u^{(i)}_r) \\
&= \frac{1}{u^{(i)}_+} (a_{s+1} \epsilon^{(i)}_{s+1} + \dots + a_r \epsilon^{(i)}_r).
\end{align*}
Since $\ubar_+ \neq 0$, we also have that
\begin{align*}
\lim_{i \rightarrow \infty} \boldsymbol\alpha \cdot \Psi(u^{(i)}) 
%&= \lim_{i \rightarrow \infty} (a_{s+1} \Psi_{s+1}(u^{(i)})+ \dots + a_r \Psi_r(u^{(i)})) \\
& = \lim_{i \rightarrow \infty} \frac{1}{u^{(i)}_+} (a_{s+1} \epsilon^{(i)}_{s+1} + \dots + a_r \epsilon^{(i)}_r) \\
&= \frac{1}{\ubar_+} \lim_{i\rightarrow\infty}(a_{s+1} \epsilon^{(i)}_{s+1} + \dots + a_r \epsilon^{(i)}_r) \\
&= 0.
\end{align*}

Furthermore, for all $i$ and $k$, $\Psi_k(u^{(i)}) > 0$.
So $\lim_{i \rightarrow \infty}\Psi_k(u^{(i)}) \geq 0$. Since each $a_i > 0$ for $i > s$, 
this implies that $\lim_{i \rightarrow \infty} \Psi_k(u^{(i)}) = 0$ for all $k > s$.

So we have that
\begin{align*}
\frac{\cbar \cdot \ubar}{\ubar_+} &= \lim_{i \rightarrow \infty} \frac{\bfc \cdot u^{(i)}}{u^{(i)}_+} \\
&= \lim_{i \rightarrow\infty} \bfc \cdot \Psi(u^{(i)}) \\
&= \cbar \cdot \Psibar(\ubar) + \sum_{j=s+1}^r c_j \big( \lim_{i \rightarrow \infty} \Psi_j(u^{(i)}) \big) \\
&= \cbar \cdot \Psibar(\ubar),
\end{align*}
as needed.
\end{proof}

\begin{proof}[Proof of Theorem \ref{thm:FacialSubsetMLE}]
First, note that $\Psibar$ is still a rational function of degree zero since deleting columns of $B$ does not affect the remaining column sums.
So $(\Bbar, \overline{\bfh})$ is a Horn pair.

By Lemma \ref{lem:InVariety}, we have that $\Psibar(\ubar) \in \overline{\Mcal_{\Abar}}$.
Since $\mathbf{1} \in \mathrm{rowspan}(\Abar)$, it follows from Lemma \ref{lem:SuffStats}
that $\sum_{k=1}^s \Psibar_k(\ubar) = 1$.
Defining a sequence $\{u^{(i)} \}_{i=1}^{\infty}$ as in the proof of Lemma \ref{lem:SuffStats},
we have that $\Psibar_k(\ubar) = \lim_{i \rightarrow\infty} \Psi_k(u^{(i)})$.
So $\Psibar_k(\ubar) \geq 0$ since each $\Psi_k(u^{(i)}) > 0$.
Furthermore, for generic choices of $\ubar$, we cannot have $\Psibar_k(\ubar) = 0$.
Indeed, for $k \leq s$, the $i$th factor of $\Psibar_k(\ubar)$ has nonzero summand $b_{ik} u_k$.
So none of these factors is zero for generic choices of $u$ of the given form.
Therefore $\Psibar(\ubar) \in \Mcal_{\Abar} = \overline{\Mcal_{\Abar}} \cap \Delta_{s-1}$. 

By Lemma \ref{lem:SuffStats},
\[
\frac{\Abar \cdot \ubar}{\ubar_+} = \Abar \cdot \Psibar(\ubar).
\]
So by Birch's theorem, $\Psibar$ is the maximum likelihood estimator for $\Mcal_{\Abar}$.
\end{proof}

Note that $\Psibar$ is a dominant map. Indeed, for generic $p \in \Mcal_{\Abar}$,
$\Psibar(p)$ is defined. Since $p$ is a probability distribution, $p_{+} = 1$.
By Birch's Theorem, $p$ is the MLE for data vector $p$.
So $\Psibar(p) = p$.

We close this section by noting that we believe that a natural 
generalization of Theorem \ref{thm:FacialSubsetMLE} is also true.

\begin{conj} 
Let $A \in \Z^{n \times r}$ and $\Abar \in \Z^{n \times s}$ a facial submatrix of $A$.
Then the maximum likelihood degree of $\Mcal_A$ is greater than or equal to the
maximum likelihood degree of $\Mcal_\Abar$.
\end{conj}

%%%%%%%%%%%%%%%%%%%%%%%%%%%%%%%%%%%%%%%%%%%%%%%%%%%%%%
%%%%%%%%%%%%%%%%%%%%%%%%%%%%%%%%%%%%%%%%%%%%%%%%%%%%%%
%%%%%%%%%%%%%%%%%%%%%%%%%%%%%%%%%%%%%%%%%%%%%%%%%%%%%%
%%%%%%%%%%%%%%%%%%%%%%%%%%%%%%%%%%%%%%%%%%%%%%%%%%%%%%

\section{Quasi-independence Models  with Non-Rational MLE}\label{sec:FacialSZ}

In this section, we show that when $S$ is not doubly chordal bipartite,
the ML-degree of $\Mcal_S$ is strictly greater than one.
We can apply Theorem \ref{thm:FacialSubsetMLE} to quasi-independence models
whose associated bipartite graphs are not doubly chordal bipartite using cycles and the following ``double square" structure.

\begin{figure}
\begin{center}
	\caption{The double-square graph associated to the matrix in Example \ref{ex:DoubleSquare}}
	 \label{Fig:DoubleSquare}
	\begin{tikzpicture}
		\draw(1,1) -- (0,1)--(0,0)--(2,0)--(2,1)--(1,1)--(1,0);
		\draw [fill] plot [only marks, mark=square*] coordinates {(0,0) (1,1) (2,0)};
		\draw [fill = white] plot [only marks, mark size=2.5, mark=*] coordinates { (0,1) (1,0) (2,1)};
		\node[below] at (0,0) {1};
		\node[above] at (1,1) {2};
		\node[below] at (2,0) {3};
		\node[above] at (0,1) {1};
		\node[below] at (1,0) {2};
		\node[above] at(2,1){3};
	\end{tikzpicture}
	\end{center}
\end{figure}
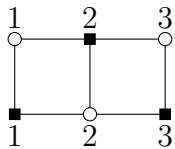
\begin{ex}\label{ex:DoubleSquare}
The minimal example of a chordal bipartite graph that is not doubly chordal bipartite is the \emph{double-square graph}.
The matrix of the double-square graph has the form
\[
\begin{bmatrix}
\star & \star & 0 \\
\star & \star & \star \\
0 & \star & \star
\end{bmatrix},
\]
or any permutation of the rows and columns of this matrix. The resulting graph, pictured in Figure \ref{Fig:DoubleSquare} is two squares joined along an edge.
This is a 6-cycle with exactly one chord and as such, is not doubly chordal bipartite.

\end{ex}

\begin{rmk}
A bipartite graph is doubly chordal bipartite if and only if it is chordal bipartite and does not have the double-square graph as an induced subgraph.
\end{rmk}

We now compute the maximum likelihood degree of models associated to the double square and to cycles of length greater than or equal to 6.

\begin{prop}\label{prop:DSMLdeg}
The maximum likelihood degree of the quasi-independence model 
whose associated graph is the double square is 2.
\end{prop}

\begin{proof}
Without loss of generality, let
\[
S = \{11, 12, 21, 22, 23, 32,33\},
\]
so that $G_S$ is a double-square graph. Then the vanishing ideal of $\Mcal_S$ is the ideal
$I(\Mcal_S)\subset \C[p_{ij} \mid ij \in S]$ given by
\[
I(\Mcal_S) = \langle p_{11}p_{22} - p_{12}p_{21}, p_{22}p_{33} - p_{23}p_{32} \rangle.
\]
Define the hyperplane arrangement
\[
\Hcal := \{ p \in \C^S \mid p_{++} \prod_{ij \in S} p_{ij} = 0 \},
\]
where $p_{++}$ denotes the sum of all the coordinates of $p$. Then Proposition 7 of \cite{amendola2019}
implies that the ML-degree of $\Mcal_S$ is the number of solutions to the system
\[
I(\Mcal_S) + \langle A(S) u + u_+ A(S) p \rangle
\]
that lie outside of $\Hcal$ for generic $u$.
Since $A(S)$ encodes the row and column marginals of $u$, 
the MLE for $u$ can be written in matrix form as
\[
\begin{bmatrix} u_{11} + \alpha & u_{12} - \alpha & 0 \\
u_{21} - \alpha & u_{22} + \alpha + \beta & u_{23} - \beta \\
0 & u_{32} - \beta & u_{33} + \beta
\end{bmatrix}
\]
for some $\alpha$ and $\beta$.
So computing the MLE is equivalent 
to solving for $\alpha$ and $\beta$ in the system
\begin{align*}
(u_{11} + \alpha)(u_{22} + \alpha + \beta) - (u_{12} - \alpha) (u_{21} - \alpha) &= 0 \\
(u_{22} + \alpha + \beta)(u_{33} + \beta) - (u_{23} - \beta) (u_{32} - \beta) &= 0.
\end{align*}
Expanding gives two equations of the form
\begin{align}\label{eqn:DoubleSquareSystem}
\alpha \beta + c_1 \alpha + c_2 \beta + c_3 &= 0 \\
\alpha \beta + d_1 \alpha + d_2 \beta + d_3 &= 0, \nonumber
\end{align}
where each $c_i, d_i$ are polynomials in the entries of $u$.

Solving for $\alpha = -(c_2 \beta + c_3)/(\beta + c_1)$ in the first equation of (\ref{eqn:DoubleSquareSystem})
and substituting into the second gives a degree 2 function of $\beta$, which can have at most two solutions.
Indeed, for generic choices of $u$, this equation has exactly two solutions, neither of which lie on $\Hcal$.
For example, take $u_{11} = u_{12} = u_{21} = u_{22} = 1$ and $u_{23} = u_{32} = u_{33} = 2$.
By performing this substitution in (\ref{eqn:DoubleSquareSystem}) with these values for $u$,
we obtain the degree 2 equation
\begin{equation}\label{eqn:DoubleSquareFxn}
\frac{-\beta^2}{\beta+4} + \frac{7 \beta}{\beta + 4} + 2\beta - 2 = 0.
\end{equation}
After clearing denominators, we obtain that $\beta^2 + 13 \beta - 8 = 0$. 
This polynomial has two distinct roots neither of which lie on $\Hcal$,
and (\ref{eqn:DoubleSquareFxn}) is defined at both of these roots.
These are generic conditions on the data; so since there exists a $u$ for which
(\ref{eqn:DoubleSquareSystem}) has exactly two solutions,
the ML-degree of $\Mcal_S$ is 2.
\end{proof}

\begin{prop}\label{prop:CycleMLdeg}
Let $S_k \subset [k] \times [k]$ be a collection of indices such that $G_{S_k}$ is a cycle of length $2k$.
Then the ML-degree of $\Mcal_{S_k}$ is $k$ if $k$ is odd and $(k-1)$ if $k$ is even.
\end{prop}

\begin{proof}
Without loss of generality, we may assume that $S_k = \{ (i,i) \mid i \in [k] \} \cup \{ (i, i+1) \mid i \in [k-1]\} \cup \{(k,1)\}.$
Since $G_{S_k}$ consists of a single cycle, the ideal $I(\Mcal_{S_k})$ is principal. Indeed, it is given by
\begin{equation}
I(\Mcal_{S_k}) = \langle \prod_{i=1}^k p_{i,i} -  \prod_{i=1}^{k}p_{i,i+1} \rangle,
\end{equation}
where we set $p_{k,k+1} = p_{k,1}$.
Let $\Hcal$ be the hyperplane arrangement,
\[
\Hcal = \{ p \mid p_{++} \prod_{ij \in S} p_{ij} = 0 \}.
\]
By Proposition 7 of \cite{amendola2019}, ML-degree of $\Mcal_{S_k}$ is the number of solutions to
\begin{equation}\label{eqn:CycleMLdeg}
I(\Mcal_{S_k}) + \langle A(S_k)u - u_+ A(S_k) p \rangle.
\end{equation}
that lie outside of $\Hcal$.

The sufficient statistics of $u$ are of the form $u_{i,i} + u_{i,i+1}$ and 
$u_{i-1,i} + u_{i,i}$ where we set $u_{0,1} = u_{k,1}$.
So computing solutions to Equation (\ref{eqn:CycleMLdeg}) is equivalent to
solving for $\alpha \in \C$ in the equation
\begin{equation}\label{eqn:CyclePolynomial}
\prod_{i=1}^k (u_{i,i} + \alpha) -  \prod_{i=1}^{k} (u_{i,i+1} - \alpha) = 0.
\end{equation}
The MLE is then of the form $p_{i,i} = (u_{i,i} + \alpha)/u_{++}$ and $p_{i,i+1} = (u_{i,i+1} - \alpha)/u_{++}$.
The degree of this polynomial is $k$ when $k$ is odd and $k-1$ when $k$ is even.

Furthermore, we claim that for generic $u$, none of these solutions lie in $\Hcal$.
Indeed, without loss of generality, suppose that $\bar{p}$ is a solution to (\ref{eqn:CycleMLdeg})
with $\bar{p}_{1,1} = 0$. Then we have that $\alpha = -u_{1,1}$.
So the first term of  (\ref{eqn:CyclePolynomial}) is 0.
But then there exists an $i$ such that
\[
u_{i,i+1} - \alpha = u_{i,i+1} + u_{1,1} = 0,
\]
which is a non-generic condition on $u$.
Similarly, since $u$ is generic, we may assume that $u_{++} \neq 0$.
But if $\bar{p}_{++} = 0$, then since each $\bar{p}_{i,i} = (u_{i,i} + \alpha) / u_{++}$ and
$\bar{p}_{i,i+1} = (u_{i,i+1} - \alpha) / u_{++}$, this implies that $u_{++} = 0$, which is a contradiction.
So for generic values of $u$, the roots of (\ref{eqn:CyclePolynomial}) give rise to exactly $k$, resp. $k-1$,
solutions to (\ref{eqn:CycleMLdeg}) 
that lie outside of $\Hcal$. So the ML-degree of $\Mcal_{S_k}$ is $k$ if $k$ is odd and $k-1$ if $k$ is even.
\end{proof}

\begin{thm}\label{thm:NotDCB}
Let $S$ be such that $G_S$ is not doubly chordal bipartite. 
Then $\Mcal_S$ does not have rational MLE.
\end{thm}

\begin{proof}
Suppose that $G_S$ is not doubly chordal bipartite.
Then it has an induced subgraph $H$ that is either a double square or a cycle of length greater than or equal to 6.
Without loss of generality, let the edge set $E(H)$ be a subset of $[k] \times [k]$.
Let $A = A(S)$ and let $\Abar$ be the submatrix of $A$ consisting of the columns indexed by elements of $E(H)$.

Let the coordinates of $P_A$ and $P_{\Abar}$ be indexed by $(x_1, \dots, x_m, y_1, \dots, y_n)$.
We claim that $\Abar$ is a facial submatrix of $A$. Indeed, $\Abar$ consists of exactly  the
 vertices of $P_A$ that satisfy $x_i = 0$ for $k < i \leq m$ and $y_j = 0$ for $k < j \leq n$.
 Since $P_A$ is a 0/1 polytope, the inequalities $x_i \geq 0$ and $y_j \geq 0$ are valid. So this constitutes a face of $P_A$.
 
 Therefore, by Propositions \ref{prop:DSMLdeg} and \ref{prop:CycleMLdeg}, 
 $A$ has a facial submatrix $\Abar$ such that $\Mcal_{\Abar}$ has ML-degree strictly greater than 1.
 So by Theorem \ref{thm:FacialSubsetMLE}, the ML-degree of $\Mcal_A = \Mcal_S$ is also strictly greater than 1,
 as needed.
\end{proof}

%%%%%%%%%%%%%%%%%%%%%%%%%%%%%%%%%%%%%%%%%%%%%%%%%%%%%%%%%%
%%%%%%%%%%%%%%%%%%%%%%%%%%%%%%%%%%%%%%%%%%%%%%%%%%%%%%%%%%
%%%%%%%%%%%%%%%%%%%%%%%%%%%%%%%%%%%%%%%%%%%%%%%%%%%%%%%%%%
%%%%%%%%%%%%%%%%%%%%%%%%%%%%%%%%%%%%%%%%%%%%%%%%%%%%%%%%%%

\section{The Clique Formula for the MLE}\label{sec:Cliques}

In this section we state the main result of the paper, which gives
the specific form of the rational maximum likelihood estimates for 
quasi-independence models when they exist.  These
are described in terms of the complete bipartite subgraphs of the associated graph $G_S$.
A complete bipartite subgraph of $G_S$ corresponds to an entirely nonzero submatrix of $S$. 
This motivates our use of the word ``clique" in the following definition.

\begin{defn}
A set of indices $C = \{i_1, \dots, i_r\} \times \{j_1, \dots, j_s\}$ is a \emph{clique} in $S$ if $(i_{\alpha},j_{\beta}) \in S$ for all $1 \leq \alpha \leq r$ and $1 \leq \beta \leq s$. A clique $C$ is maximal if it is not contained in any other clique in $S$.
\end{defn}

We now describe some important sets of cliques in $S$.

\begin{notation}\label{not:MaxInt}
For every pair of indices $(i,j) \in S$, we let $\Max(ij)$ be the set of all maximal cliques in $S$ that contain $(i,j)$. 
We let $\Int(ij)$ be the set of all containment-maximal pairwise intersections of elements of $\Max(ij)$. 
Similarly, we let $\Max(S)$ denote the set of all maximal cliques in $S$
and $\Int(S)$ denote the set of all maximal intersections of maximal cliques in $S$. 
\end{notation}

\begin{ex}\label{ex:RunningExample}
Let $m = 8$ and $n = 9$. Consider the set of indices
\[
S = \{11, 12, 21, 22, 23, 28, 31,32,33,34,41,45,51,56,57,65,76,86,87, 89 \},
\]
where we replace $(i,j)$ with $ij$ for the sake of brevity. The corresponding matrix with structural zeros is
\[
\begin{bmatrix}
\star & \star & 0 & 0 & 0 & 0 & 0 & 0 & 0 \\
\star & \star & \star & 0 & 0 & 0 & 0 & \star & 0 \\
\star & \star & \star & \star & 0 & 0 & 0 & 0 & 0\\
\star & 0 & 0 & 0 & \star & 0 & 0 & 0 & 0\\
\star & 0 & 0 & 0 & 0 & \star & \star & 0 & 0\\
0 & 0 & 0 & 0 & \star & 0 & 0 & 0 & 0\\
0 & 0 & 0 & 0 & 0 & \star & 0 & 0 & 0\\
0 & 0 & 0 & 0 & 0 & \star & \star & 0 & \star
\end{bmatrix}.
\]
We will use this as a running example. The bipartite graph $G_S$ associated to $S$ is pictured in Figure \ref{Fig:RunningGraph}. 
In this figure, we use white circles to denote vertices corresponding to rows in $S$ and black squares to denote vertices corresponding to columns in $S$. Note that $G_S$ is doubly chordal bipartite since its only cycle of length 6 has two chords.

In this case, the set of maximal cliques in $S$ is
\begin{align*}
\Max(S) =& \big\{ \{11, 21, 31, 41, 51\}, \{11, 12, 21, 22, 31, 32\}, \{21, 22, 23, 31, 32, 33\}, \{21, 22, 23, 28\}, \\
	& \quad \{31, 32, 33, 34\}, \{41, 45\}, \{51, 56, 57\}, \{45,65\}, \{56,76, 86\}, \{56,57,86,87\}, \{86, 87, 89\}  \big\}.
\end{align*}
The set of maximal intersections of maximal cliques in $S$ is
\begin{align*}
\Int(S) =& \big\{ \{11, 21, 31\}, \{21, 22, 31,32\}, \{21, 22, 23\}, \{31,32,33\}, \{41\}, \{51\}, \{45\}, \\
 & \quad \{56,57\}, \{56,86\}, \{86,87\} \big\}.
\end{align*}
Note, for example, that $\{31,32\}$ is the intersection of the two maximal cliques $\{11, 12, 21, 22, 31, 32\}$ and $\{31, 32, 33, 34\}$. 
However it is not in $\Int(S)$ because it is properly contained in the intersection of maximal cliques,
\[
\{11, 12, 21, 22, 31, 32\} \cap \{21, 22, 23, 31, 32, 33\} = \{21, 22, 31, 32\}.
\]

\begin{figure}
\begin{center}
	\begin{tikzpicture}
		\draw (5,0) -- (5,1)--(6.5,2) -- (6.5,3)--(5,4)--(3.5,3) -- (3.5,2) -- (5,1) -- (5,4)--(7,5)--(10,5);
		\draw(3.5,2)--(6.5,3);
		\draw(0,5)--(3,5)--(3,4)--(1.5,4)--(1.5,5);
		\draw(3,5)--(5,4);
		\draw (1.5,4)--(1.5,3);
		\draw(6.5,3)--(8,3);
		\draw [fill] plot [only marks, mark=square*] coordinates {(5,4) (3.5,2) (6.5,2) (5,0) (8.5,5) (1.5,5) (3,4) (1.5,3) (8,3)};
		\draw [fill = white] plot [only marks, mark size=2.5, mark=*] coordinates { (3.5,3) (6.5,3) (5,1) (7,5) (3,5) (10,5) (0,5) (1.5,4)};
		\node[left] at (3.5,3) {1};
		\node[above] at (6.5,3){2};
		\node[right] at (5,1){3};
		\node[above] at (7,5){4};
		\node[above] at (3,5){5};
		\node[above] at (10,5) {6};
		\node[above] at (0,5) {7};
		\node[left] at (1.5,4) {8};
		\node[above] at (5,4) {1};
		\node[left] at (3.5,2) {2};
		\node[right] at (6.5,2){3};
		\node[right] at (5,0){4};
		\node[above] at (8.5,5){5};
		\node[above] at (1.5,5){6};
		\node[below] at (3,4){7};
		\node[below] at(1.5,3){9};
		\node[above] at (8,3){8};
	\end{tikzpicture}
	\end{center}
	\caption{The bipartite graph associated the matrix $S$ in Example \ref{ex:RunningExample}}
	 \label{Fig:RunningGraph}	
\end{figure}
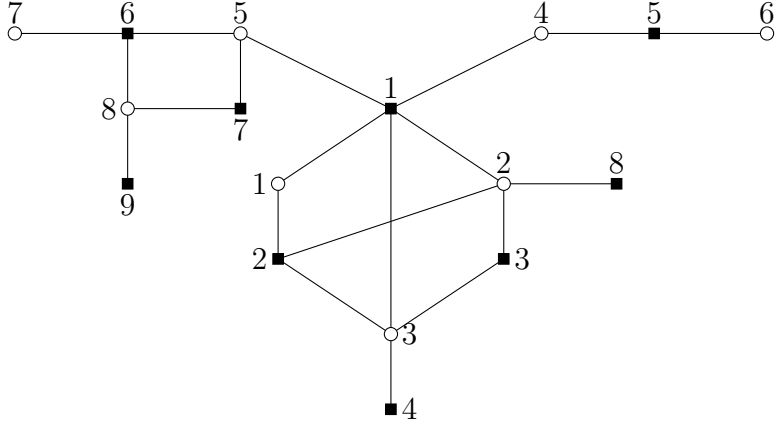
\end{ex}

Let $u = (u_{ij} \mid (i,j) \in S)$ be a matrix of counts. For any $C \subset S$, we let $C^+$ denote the sum of all the entries of $u$ whose indices are in $C$. That is,
\[
C^+ = \sum_{(i,j) \in C} u_{ij}.
\]
Similarly, we denote the row and column marginals $u_{i+} = \sum_{j: (i,j) \in S} u_{ij}$ and $u_{+j} = \sum_{i: (i,j) \in S} u_{ij}$.
The sum of all entries of $u$ is $u_{++} = \sum_{(i,j) \in S} u_{ij}$.

\begin{thm}\label{thm:Main}
Let $S \subset [m] \times [n]$ be a set of indices with associated 
bipartite graph $G_S$ and quasi-independence model  $\Mcal_S$. 
Then $\Mcal_S$ has rational maximum likelihood estimate if and 
only if $G_S$ is doubly chordal bipartite. In particular, if 
$u = (u_{ij} \mid (i,j) \in S)$ is a matrix of counts, 
the maximum likelihood estimate for $u$ has $ij$th entry
\[
\hat{p}_{ij} = \frac{u_{i+}u_{+j}\displaystyle{\prod_{C \in \Int(ij)} C^+}}{u_{++} \displaystyle{ \prod_{D \in \Max(ij)}D^+}}
\]
where the sets $\Max(ij)$ and $\Int(ij)$ are as in Notation \ref{not:MaxInt}.
\end{thm}

Over the course of the next two sections, we prove various lemmas 
that  ultimately allow us to prove Theorem \ref{thm:Main}.
\begin{ex}
	Consider the set of indices $S$ from Example \ref{ex:RunningExample}.
	Let $u$ be a matrix of counts. Consider the maximum likelihood estimate for the $(2,1)$ entry, $\hat{p}_{21}$.
	The maximal cliques that contain $21$ are $\{11, 21, 31, 41, 51\}, \{11, 12, 21, 22, 31, 32\}$,
	$\{21, 22, 23, 28\}$ and $\{21, 22, 23, 31, 32, 33\}$.
	The maximal intersections of maximal cliques that contain $21$ are $\{11, 21,31\}$, $\{21, 22, 23\}$ and $\{21, 22, 31, 32\}$.
	Since $S$ is doubly chordal bipartite, we apply Theorem \ref{thm:Main} to obtain that the numerator of $\hat{p}_{21}$ is
	\[
	(u_{21} + u_{22} + u_{23} + u_{28})(u_{11} + u_{21} + u_{31} + u_{41} + u_{51})(u_{11} + u_{21} + u_{31})(u_{21} + u_{22} + u_{23})(u_{21} + u_{22} + u_{31} + u_{32}).
	\]
	The denominator of $\hat{p}_{21}$ is {\footnotesize
	\[
	u_{++}(u_{11} + u_{21} + u_{31} + u_{41} + u_{51})(u_{11} +u_{12} + u_{21} + u_{22} + u_{31} + u_{32})(u_{21} + u_{22} + u_{23} + u_{28})(u_{21} + u_{22} + u_{23} + u_{31} + u_{32} + u_{33}).
	\]}
	We note that when a maximal clique is a single row or column, as is the case with  $\{21, 22, 23, 28\}$ and $\{11, 21, 31, 41, 51\}$, we have cancellation between the numerator and denominator.
\end{ex}

In order to prove Theorem \ref{thm:Main}, we show that $\hat{p}_{ij}$ satisfies the conditions of Birch's theorem.
First, we investigate the intersections of a fixed column of the matrix 
with structural zeros with maximal cliques and their intersections. 
We  prove useful lemmas about the form that these maximal cliques have 
that  allow us to show that the conditions of Birch's Theorem are satisfied.
In particular, we  use them to prove Corollary \ref{cor:SumEntireBlock}, which states
  that the column marginal of the formula in Theorem \ref{thm:Main}
given by the fixed column is equal to that of the normalized data.

%%%%%%%%%%%%%%%%%%%%%%%%%%%%%%%%%%%%%%%%%%%%%%%%%%%%%%%%%%%
%%%%%%%%%%%%%%%%%%%%%%%%%%%%%%%%%%%%%%%%%%%%%%%%%%%%%%%%%%%
%%%%%%%%%%%%%%%%%%%%%%%%%%%%%%%%%%%%%%%%%%%%%%%%%%%%%%%%%%%
%%%%%%%%%%%%%%%%%%%%%%%%%%%%%%%%%%%%%%%%%%%%%%%%%%%%%%%%%%%

\section{Intersections of Cliques with a Fixed Column}\label{sec:FixedColumn}

In this section we prove some results that will set the stage for the proof
of Theorem \ref{thm:Main} that appears in Section \ref{sec:BirchsThm}.
To prove that our formulas satisfy Birch's theorem, we need to understand
what happens to sums of these formulas over certain sets of indices.

Let $S \subset [m] \times [n]$ and let $j_0 \in [n]$.
 Without loss of generality, we assume that $(1,j_0),\dots,(r,j_0) \in S$, 
 and that the last $(i,j_0) \not\in S$ for all $i > r$.
Let 
\[
N_{j_0} := \{ (1, j_0), \dots, (r, j_0) \}.
\]
 We consider $j_0$ to be the index of a column in the matrix representation of $S$, and $1, \dots, r$ to be the indices of its nonzero rows.
Now let $T_0 \mid \dots \mid T_h$ be the coarsest partition of $[n]$ with the property that whenever $j,k \in T_{\ell}$,
\[
\{ i \in [r] \mid (i,j) \in S \} = \{ i \in [r] \mid (i,k) \in S\}.
\]
In the matrix representation of $S$, each $T_{\ell}$ corresponds to a set of columns whose first $r$ rows are identical. The fact that we take $T_0 \mid \dots \mid T_h$ to be the coarsest such partition ensures that the supports of the columns in distinct parts of the partition are distinct.

Define the partition $B_0 \mid \dots \mid B_h$ of $S \cap ([r] \times [n])$ by 
$B_{\ell} = \{(i,j) \mid j \in T_{\ell}\}$. Note that one of the $B_{\ell}$ may be empty, 
in which case we exclude it from the partition.
We call these $B_{\ell}$ the \emph{blocks} of $S$ corresponding to column $j_0$.
We fix $j_0$ and $B_0, \dots, B_h$ for the entirety of this section, 
and we assume without loss of generality that $j_0 \in T_0$.

Denote by $\rows(B_{\alpha})$ the set of all $i \in [r]$ such that $(i,j) \in B_{\alpha}$ for some column index $j$. 
Note that this is a subset of the first $1, \dots, r$ rows of $S$, 
and that in the matrix representation of $S$, the columns whose indices are in $B_{\alpha}$ may not have the same zero patterns in rows $r+1, \dots m$.
Similarly, for each $j \in [n]$, define $\rows(j)$ to be the set of all $i \in [r]$ such that $(i,j) \in S$; 
that is, the elements of $\rows(j)$ are the row indices of the nonzero entries of column $j$ in the
first $r$ rows of $S$.  Note that the dependence on $j_0$ in this notation stems from the
fact that the column $j_0$ is used to obtained the partition $B_0 \mid \dots \mid B_h$.

\begin{ex}\label{ex:BlockLabels}
Consider the running example $S$ from Example \ref{ex:RunningExample}, and let $j_0 = 1$ be the first column of $S$.
In this case, $r = 5$ since only the first 5 rows entries of column $j_0$ are nonzero.
Then the blocks associated to $j_0$ consist of the following columns.
\begin{align*}
T_0 & = \{j_0\} = \{1\} & \qquad & T_1 = \{2\} \\
T_2 &= \{3\} & \qquad & T_3 = \{4\} \\
T_4 &= \{5\} & \qquad & T_5 = \{6,7\} \\
T_6 &= \{8\} & \qquad &T_7 = \{9\}.
\end{align*}

We note that although columns 6 and 7 are not the same over the whole matrix, their first five rows are the same. 
Since these are the nonzero rows of column $j_0$, columns 6 and 7 belong to the same block.

The $B_i$ associated to each of these sets of column indices are
\begin{align*}
B_0 &= \{11, 21, 31, 41, 51\} & \qquad B_1 &= \{12, 22, 32\} \\
B_2 &= \{23,33\} & \qquad B_3 &= \{34\} \\
B_4 &= \{45\} & \qquad B_5 &= \{56, 57\} \\
B_6 &= \{28\} & \qquad B_7 &= \emptyset
\end{align*}

For instance, $\rows(B_1) = \{1,2,3\}$ and $\rows(B_5) = \{5\}$.
\end{ex}

The following proposition characterizes what configurations of the rows of the $B_{\alpha}$'s
 are allowable in order to avoid a cycle with exactly one chord. We call the condition outlined in Proposition \ref{prop:DSFree} the double-squarefree, 
or \emph{DS-free} condition.

\begin{prop}[DS-free condition]\label{prop:DSFree}
Let $S$ be doubly chordal bipartite. 
Let $\alpha, \beta \in [h]$. If $\rows(B_{\alpha}) \cap \rows(B_{\beta})$ is nonempty, 
then $\rows(B_{\alpha}) \subset \rows(B_{\beta})$ or $\rows(B_{\beta}) \subset \rows(B_{\alpha})$.
\end{prop}

\begin{proof}
For the sake of contradiction, suppose without loss of generality that $\rows(B_1) \cap \rows(B_2)$ is nonempty
but neither is contained in the other.
Then let $i_0, i_1 \in \rows(B_1)$ and $i_1, i_2 \in \rows(B_2)$
so that $i_0 \not\in \rows(B_2)$ and $i_2 \not\in \rows(B_1)$.
We have $i_0, i_1, i_2 \in \rows(j_0)$ by definition.
Let $j_1 \in T_1$ and $j_2 \in T_2$.
Then the $\{i_0, i_1, i_2\} \times \{j_0, j_1, j_2\}$ submatrix of $S$ is the matrix of a double-square,
which contradicts that $S$ is doubly chordal bipartite.
\end{proof}

Proposition \ref{prop:DSFree} implies that the sets $\rows(B_{\alpha})$ over 
all $\alpha$ have a tree structure ordered by containment.
In fact, we will see that this gives a tree structure on 
the maximal cliques in $S$ that intersect $N_{j_0}$.
(Recall that $N_{j_0} = \{ (i, j_0) \in S \} = [r] \times \{j_0\}$).

\begin{ex}
The matrix $S$ from Example \ref{ex:RunningExample} is doubly chordal bipartite, and as such, satisfies the DS-free condition.
If we append a tenth column, $(0, \star , \star ,\star, 0, 0, 0, 0)^T$ to obtain a matrix $S'$, this introduces a new block $B_8$ which just contains column 10. 
This matrix violates the DS-free condition
since $\rows(B_1) = \{1,2,3\}$ and $\rows(B_8) = \{2,3,4\}$. Their intersection is nonempty, but neither is contained in another.
Indeed, the $\{1,2,4\} \times \{1,2,10\}$ submatrix of $S'$ is the matrix of a double-square.
\end{ex}

For each pair of indices $ij$ such that $(i,j) \in S$, 
let $x_{ij}$ be the polynomial obtained from $\hat{p}_{ij}$ by simultaneously clearing the denominators of all $\hat{p}_{k\ell}$.
That is, to obtain $x_{ij}$, we multiply $\hat{p}_{ij}$ by $u_{++} \prod_{D \in \Max(S)} D^+$ so that
\[
x_{ij} = u_{i+}u_{+j}\displaystyle{\prod_{C \in \Int(ij)} C^+}\displaystyle{ \prod_{D \in \Max(S) \setminus \Max(ij)}D^+}.
\]
Our main goal in this section is to derive a formula for the sum,
\[
\sum_{i \in \rows(B_{\alpha})} x_{ij_0}.
\]
This is the content of 
Lemma \ref{lem:SumInBlock}. This formula allows us to verify
that the $j_0$ column marginal of $\hat{p}$ matches that of the normalized data.
In order to simplify this sum,
we must first understand how maximal cliques and their intersections intersect $N_{j_0}$. 
For each $B_{\alpha}$ with $\alpha \in [h]$ and $\rows(B_{\alpha}) \neq \emptyset$, we let $D_{\alpha}$ be the clique,
\[
D_{\alpha} = \{ (i,j) \mid i \in \rows(B_{\alpha}) \text{ and } \rows(B_{\alpha}) \subset \rows(j) \}.
\]
In other words, $D_{\alpha}$ is the largest clique that contains $B_{\alpha}$ and intersects $N_{j_0}$.
We call $D_{\alpha}$ the \emph{clique induced by} $B_{\alpha}$. 

\begin{ex}\label{ex:MaximalCliques}
Consider our running example $S$ with $j_0 = 1$ and blocks $B_0, \dots, B_7$ as described in Example \ref{ex:BlockLabels}.
Then $N_{j_0} = N_1 = \{11, 21, 31, 41, 51\}.$
The cliques induced by $B_0, \dots, B_7$ are
\begin{align*}
D_0 & = \{11, 21, 31, 41, 51\}, \\
D_1 &= \{11, 12, 21, 22, 31, 32\}, \\
D_2 &=  \{21, 22, 23, 31, 32, 33\}, \\
D_3 &=  \{31, 32, 33, 34\}, \\
D_4 &= \{41, 45\}, \\
D_5 &=  \{51, 56, 57\}, \text{ and }\\
D_6 &= \{21, 22, 23, 28\}.
\end{align*}
There is no $D_7$ since the block $B_7$ is empty. Note that these are exactly the maximal cliques in $S$ that intersect $N_1$. The next proposition proves that this is the case for all DS-free matrices with structural zeros.
\end{ex}

We note that when $D_{\alpha}$ is the clique induced by $B_{\alpha}$,
all of the nonzero rows of $D_{\alpha}$ lie in $[r]$ by definition of an induced clique.
We continue to use the notation $\rows(D_{\alpha})$ since the formation
of the set $B_{\alpha}$ depends on the specified column $j_0$.
For any clique $C$, let $\cols(C) = \{ j \mid (i,j) \in C \text{ for some } j\}$.
 
\begin{prop}\label{prop:MaxCliqueCharacterization}
For all $\alpha \in [h]$, $D_{\alpha}$ is a maximal clique. 
Furthermore, any maximal clique that has nonempty intersection with $N_{j_0}$ is induced by some $B_{\alpha}$.
\end{prop}

\begin{proof}
We will show that $D_{\alpha}$ is maximal by showing that we cannot add any rows or columns to it.
We cannot add any columns to $D_{\alpha}$ by definition.
We cannot add any of rows $1, \dots, r$ to $D_{\alpha}$ since all nonzero rows of $B_{\alpha}$ are already contained in $D_{\alpha}$.
We cannot add any of rows $r+1, \dots, m$ to $D_{\alpha}$ since $j_0$ is a column of $D_{\alpha}$ whose entries in rows $r+1, \dots, m$ 
are zero.
Note that if we can add one element $(i,j)$ to $D_{\alpha}$, then by definition of a clique,
we must either be able to add all of $\{i\} \times \cols(D_{\alpha})$ or $\rows(D_{\alpha}) \times \{j\}$ to the clique.
Since we cannot add any rows or columns to $D_{\alpha}$, it is a maximal clique.

Now let $D$ be a maximal clique that intersects $N_{j_0}$.
For the sake of contradiction, suppose that $D \neq D_{\alpha}$ for each $\alpha \in [h]$.

Let $j_1$ be a column in $D$ such that $\rows(j_1)$ is minimal among all columns of $D$.
We must have that $(i,j_1) \in B_{\alpha}$ for some $i \in [r]$ and $\alpha \in [h]$.
Since $D \neq D_{\alpha}$, it must be the case that column $j_1$ has a nonzero row $i_1 \in [r]$ that is not in $D$.
Since $D$ is maximal, there must exist another column $j_2$ in $D$ that has a zero in row $i_1$.
Therefore, we have that $\rows(j_1) \not\subset \rows(j_2)$.
Furthermore, $\rows(j_2) \not\subset \rows(j_1)$ by the minimality of $j_1$.
But since $D$ is nonempty, the intersection of $\rows(j_1)$ and $\rows(j_2)$ must be nonempty.
This contradicts Proposition \ref{prop:DSFree}, as needed.
%Since $D$ is nonempty, columns $j_1$ and $j_2$ have a row $i_3$ in common.=
%Then the $\{i_1, i_2, i_3\} \times \{j_0,j_1, j_2\}$ submatrix of $S$ is the matrix of a double square,
%and we have reached a contradiction.
\end{proof}

Proposition \ref{prop:MaxCliqueCharacterization} shows that the maximal cliques
 that intersect $N_{j_0}$ are exactly the cliques that are induced by some $B_{\alpha}$. 
 The DS-free condition gives a poset structure on the set of these maximal cliques 
 $D_0, \dots, D_h$ that intersect $N_{j_0}$ nontrivially. 
 
 \begin{defn}
 Let $P(j_0)$ denote the poset with ground set $\{D_0, \dots D_h\}$ 
 and $D_{\alpha} \leq D_{\beta}$ if and only if $\rows(D_{\alpha}) \subset \rows(D_{\beta})$.
 \end{defn}
 
 Recall that for a poset $P$ and two elements of its ground set, $p, q \in P$,
 we say that $q$ \emph{covers} $p$ if $p < q$ and for any $r \in P$,
 if $p \leq r \leq q$, then $r = p$ or $r = q$. We denote such a cover relation by $p \lessdot q$.
 The \emph{Hasse diagram} of a poset is a directed acyclic graph on $P$ with an edge
 from $p$ to $q$ whenever $p \lessdot q$.
 In the case of $P(j_0)$, the Hasse diagram of this poset is a tree since the DS-free condition implies
 that any $D_{\alpha}$ is covered by at most one maximal clique. 
 
 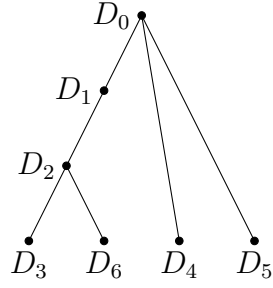
\begin{figure}
 \begin{center}
	\begin{tikzpicture}
	\draw (0,0) -- (1.5,3) -- (3,0);
	\draw(.5,1)--(1,0);
	\draw (1.5,3)--(2,0);
	\draw[fill] (0,0) circle [radius = .05];
	\node[below] at (0,0) {$D_3$};
	\draw[fill] (1,0) circle [radius = .05];
	\node[below] at (1,0) {$D_6$};
	\draw[fill] (2,0) circle [radius = .05];
	\node[below] at (2,0) {$D_4$};
	\draw[fill] (3,0) circle [radius = .05];
	\node[below] at (3,0) {$D_5$};
	\draw[fill] (.5,1) circle [radius = .05];
	\node[left] at (.5,1) {$D_2$};
	\draw[fill] (1,2) circle [radius = .05];
	\node[left] at (1,2) {$D_1$};
	\draw [fill] (1.5,3) circle[radius = .05];
	\node[left] at (1.5,3) {$D_0$};
	\end{tikzpicture}
	\end{center}
	 	\caption{The poset $P(j_0)$ for Example \ref{ex:Poset}}
	 \label{Fig:Poset}
 \end{figure}
 
 \begin{ex}\label{ex:Poset}
 In our running example $S$ with $j_0 = 1$ and blocks $B_0, \dots B_6$ and associated cliques $D_0, \dots, D_6$, the Hasse diagram of the poset $P(j_0)$ is pictured in Figure \ref{Fig:Poset}.
 \end{ex}

The next proposition shows that the cover relations in this poset, denoted $D_{\alpha} \lessdot D_{\beta}$
 correspond to maximal intersections of maximal cliques that intersect $N_{j_0}$ nontrivially.
Denote by $\cols(D_{\alpha})$ the nonzero columns of the clique $D_{\alpha}$.
We note that if $D_{\alpha} \lessdot D_{\beta}$, then $\cols(D_{\beta}) \subset \cols(D_{\alpha})$.
In particular, this means that if $C = D_{\alpha} \cap D_{\beta}$, then
$C = \rows(D_{\alpha}) \times \cols(D_{\beta})$.

\begin{prop}\label{prop:MaxIntCharacterization}
Let $C = D_{\alpha} \cap D_{\beta}$. Then $C$ is maximal among all pairwise intersections of maximal cliques if and only if $D_{\alpha} \lessdot D_{\beta}$ or $D_{\beta} \lessdot D_{\alpha}$ in $P(j_0)$.
\end{prop}

\begin{proof}
Suppose without loss of generality that $D_{\alpha} \lessdot D_{\beta}$ in $P(j_0)$. 
For the sake of contradiction, suppose that $D_{\alpha} \cap D_{\beta} \not\in \Int(S)$.
Then there exists another maximal clique that contains $D_{\alpha} \cap D_{\beta}$.
By Proposition \ref{prop:MaxCliqueCharacterization} and the fact that $D_{\alpha} \cap D_{\beta}$
intersects $N_{j_0}$ nontrivially, we can write this maximal clique as $D_{\gamma}$ for some $\gamma \in [h]$.

Note that we have $\rows(C) = \rows(D_{\alpha})$ and $\cols(C) = \cols(D_{\beta})$.
Therefore $C = \rows(D_{\alpha}) \times \cols(D_{\beta})$.
So $\rows(D_{\alpha}) \subsetneq \rows(D_{\gamma})$ and $\cols(D_{\beta}) \subsetneq \cols(D_{\gamma})$.
In particular, this second inclusion implies that $\rows(D_{\gamma}) \subsetneq \rows(D_{\beta})$.
Indeed, suppose that $i \leq r$ is a row of $D_{\gamma}$ that is not a row of $D_{\beta}$.
Then there exists a column $j$ of $D_{\beta}$ for which $(i,j) \not\in S$.
But since $j$ is also a column of $D_{\gamma}$, this contradicts that $D_{\gamma}$ is a clique.
So we have the proper containments
\[
\rows(D_{\alpha}) \subsetneq \rows(D_{\gamma}) \subsetneq \rows(D_{\beta}),
\]
which contradicts that $D_{\alpha} \lessdot D_{\beta}$ in $P(j_0)$.
So $D_{\alpha} \cap D_{\beta}$ must be maximal.
%
%Since $\rows(D_{\alpha}) \subset \rows(D_{\beta})$, we have that $\rows(D_{\alpha}) \cap \rows(D_{\beta}) = \rows(D_{\alpha})$.
%So we must have that $\rows(D_{\alpha}) \subsetneq \rows(D_{\gamma})$.
%
%
%Furthermore, since $D_{\alpha} \cap D_{\beta} \subset D_{\gamma}$, 
%and the $\cols(B_{\beta}) \subset D_{\alpha} \cap D_{\beta}$,
%we must have that $\rows(D_{\gamma}) \subsetneq \rows(D_{\beta})$.
%But this contradicts that $D_{\alpha} \lessdot D_{\beta}$ in $P(j_0)$.
%So $D_{\alpha} \cap D_{\beta}$ must be maximal.

Now let $C = D_{\alpha} \cap D_{\beta} \in \Int(S)$.
For the sake of contradiction, suppose that $D_{\alpha}$ does not cover $ D_{\beta}$ or vice versa.
Since $C$ is nonempty, without loss of generality we must have 
$\rows(D_{\alpha}) \subset \rows(D_{\beta})$ by the DS-free condition.
So there exists a $D_{\gamma}$ such that $D_{\alpha} < D_{\gamma} < D_{\beta}$ in $P(j_0)$.
Therefore we have that
\[
\rows(D_{\alpha}) \subsetneq \rows(D_{\gamma}) \subsetneq \rows(D_{\beta}).
\]

Let $(i,j) \in C$. Then $i$ is a row of $D_{\alpha}$, so it is a row of $D_{\gamma}$.
Furthermore, since $j$ is a column of $D_{\beta}$, $\rows(D_{\beta}) \subset \rows(j)$.
So $\rows(D_{\gamma}) \subset \rows(j)$ and $j$ is a column of $D_{\gamma}$.
Therefore, $C \subsetneq D_{\gamma} \cap D_{\beta}$.
This containment is proper since $\rows(D_{\alpha}) \subsetneq \rows(D_{\gamma})$.
So we have contradicted that $C$ is maximal.
\end{proof}

We can now state the key lemma regarding the sum of the $x_{ij}$s over 
$\{ i:  (i,j_0)  \in D_{\alpha} \}$ for any $\alpha \in [h]$.

\begin{lemma}\label{lem:SumInBlock}
Let $S$ be DS-free and let $D_{\alpha}$ be a maximal clique that intersects $N_{j_0}$. Then
\begin{equation}\label{eqn:SumInBlock}
\sum_{i \in \rows(D_{\alpha})} x_{ij_0} = u_{+j_0} 
\Big( \prod_{\substack{C \in \Int(S) \\ C \cap N_{j_0} \neq \emptyset \\ \rows(D_{\alpha}) \subset \rows(C) }} C^+ \Big)
\Big( \prod_{\substack{D \in \Max(S) \\ D \cap N_{j_0} \neq \emptyset \\ \rows(D) \subset \rows(D_{\alpha})}} D^+ \Big)
\Big( \prod_{\substack{E \in \Max(S) \\ N_{j_0} \cap D_{\alpha} \cap E = \emptyset}} E^+ \Big)
\end{equation}
\end{lemma}

In order to prove this, we will sum the entries $x_{ij_0}$ over all $i \in \rows( D_{\beta})$ for each $\beta$. We will do this inductively from the bottom of $P(j_0)$.
The key idea of this induction is as follows. 

\begin{rmk}\label{rmk:InductionIdea}
If $D_{\alpha_1}, \dots, D_{\alpha_{\ell}}$ are covered by $D_{\beta}$ in $P(j_0)$, 
then the rows of $N_{j_0} \cap D_{\beta}$ are partitioned by each $N_{j_0} \cap D_{\alpha_k}$
along with the set of rows that are in $D_{\beta}$ and not in any $D_{\alpha_k}$. The fact that this is a partition follows from the DS-free condition. 
Therefore, summing the $x_{ij_0}$ that belong to each clique covered by $D_{\beta}$ and adding in the $x_{ij_0}$s for rows $i$ that are not in any clique covered by $D_{\beta}$
will give us the sum of $x_{ij_0}$ over all $i \in \rows(D_{\beta})$.
\end{rmk}

The next proposition focuses on the factors of the right-hand side of Equation (\ref{eqn:SumInBlock}) that correspond to elements of $\Int(S)$.
It will be used to show that when we perform the induction and move upwards by one cover relation from $D_{\alpha}$ to $D_{\beta}$ in the poset $P(j_0)$,
all but one of these factors stays the same. The only one that no longer appears in the product corresponds to the maximal intersection $D_{\alpha} \cap D_{\beta}$.

\begin{prop}\label{prop:IntersectionsFactorOut}
Let $D_{\alpha} \lessdot D_{\beta}$ in $P(j_0)$. 
Let $C \in \Int(S)$ intersect $N_{j_0}$ nontrivially so that $\rows(D_{\alpha}) \subset \rows(C)$. 
Then either $C = D_{\alpha} \cap D_{\beta}$ or $\rows(D_{\beta}) \subset \rows(C)$.
\end{prop}

\begin{proof}
Without loss of generality, let $C = D_1 \cap D_2$. 
Proposition \ref{prop:MaxIntCharacterization} tells us that $C$ must be of this form. 
By the same proposition, we may assume without loss of generality that $D_2 \lessdot D_1$,
so $\rows(C) = \rows(D_2)$.
Suppose that $\rows(D_{\beta}) \not\subset \rows(D_2)$. 
Since $\rows(D_{\alpha}) \subset \rows(D_2)$ and $\rows(D_{\alpha}) \subset \rows(D_{\beta})$,
we must have that $\rows(D_2) \cap \rows(D_{\beta})$ is nonempty.
So by the DS-free condition, $\rows(D_2) \subsetneq \rows(D_{\beta})$.
So we have the chain of inclusions,
\[
\rows(D_{\alpha}) \subset \rows(D_2) \subsetneq \rows(D_{\beta}).
\]
But since $D_{\beta}$ covers $D_{\alpha}$ in $P(j_0)$,
and every element of $P(j_0)$ is covered by at most one element,
this implies that $\alpha = 1$ and $\beta = 2$, so $C = D_{\alpha} \cap D_{\beta}$, as needed.
\end{proof}

The following proposition focuses on the factors of the right-hand side of Equation (\ref{eqn:SumInBlock}) that correspond to elements of $\Max(S)$.
It gives a correspondence between the factors of this product for $D_{\alpha}$ and all but one of the factors of this product for $D_{\beta}$ when we have the cover relation $D_{\alpha} \lessdot D_{\beta}$ in $P(j_0)$.

\begin{prop}\label{prop:MaxCliquesFactorOut}
For any $D_{\alpha}, D_{\beta} \in P(j_0)$, define the following sets:
\begin{align*}
R_{\alpha} &= \{D\in \Max(S) \mid D \cap N_{j_0} \neq \emptyset, \rows(D) \subset \rows(D_{\alpha}) \}  \cup \{E \in \Max(S) \mid N_{j_0} \cap D_{\alpha} \cap E = \emptyset \} \\
\overline{R}_{\beta} &= \{D \in \Max(S) \mid D \cap N_{j_0} \neq \emptyset, \rows(D) \subsetneq \rows(D_{\beta}) \} \cup \{E \in \Max(S) \mid N_{j_0} \cap D_{\beta} \cap E = \emptyset \}.
\end{align*}
If $D_{\alpha} \lessdot D_{\beta}$ in $P(j_0)$, then $R_{\alpha} = \overline{R}_{\beta}$.
\end{prop}

\begin{proof}
First let $D \in R_{\alpha}$.
If $\rows(D) \subset \rows(D_{\alpha})$ and $D \cap N_{j_0} \neq \emptyset$, then since $\rows(D_{\alpha}) \subsetneq \rows(D_{\beta})$,
we have that $\rows(D) \subsetneq \rows(D_{\beta})$.
So $D \in \overline{R}_{\beta}$.

Otherwise, we have $N_{j_0} \cap D_{\alpha} \cap D = \emptyset$. There are now two cases.

\emph{Case 1:} If $N_{j_0} \cap D = \emptyset$, then $N_{j_0} \cap D_{\beta} \cap D = \emptyset$ as well.
So $D \in \overline{R}_{\beta}$.

\emph{Case 2:} Suppose that $N_{j_0} \cap D \neq \emptyset$ and $D_{\alpha} \cap D = \emptyset$.
If $D_{\beta} \cap D$ is empty as well, then $D \in \overline{R}_{\beta}$.

Otherwise, suppose $D_{\beta} \cap D \neq \emptyset$.
Then we must have that $\rows(D) \subsetneq \rows(D_{\beta})$
by the fact that $\rows(D_{\beta}) \not\subset \rows(D)$ and the DS-free condition.
So $D \in \overline{R}_{\beta}$ in this case as well.
Note that it is never the case that $N_{j_0} \cap D \neq \emptyset$ and
 $D_{\alpha} \cap D \neq \emptyset$ but $N_{j_0} \cap D \cap D_{\alpha} = \emptyset$
since $j_0$ is a column of $D_{\alpha}$.
So we have shown that $R_{\alpha} \subset \overline{R}_{\beta}$.

Now let $D \in \overline{R}_{\beta}$. We have two cases again.

\emph{Case 1:} First, consider the case in which $\rows(D) \subsetneq \rows(D_{\beta})$ and $D \cap N_{j_0} \neq \emptyset$.
If $\rows(D) \subset \rows(D_{\alpha})$, then $D \in R_{\alpha}$, as needed.
Otherwise, by the DS-free condition, there are two cases.

\emph{Case 1a:} If $\rows(D_{\alpha}) \subsetneq \rows(D)$, then we have the chain of containments,
\[
\rows(D_{\alpha}) \subsetneq \rows(D) \subsetneq \rows(D_{\beta}),
\]
which contradicts that $D_{\alpha} \lessdot D_{\beta}$ in $P(j_0)$. So this case cannot actually occur.

\emph{Case 1b:} If $\rows(D_{\alpha}) \cap \rows(D) = \emptyset $,
then we have that $N_{j_0} \cap D_{\alpha} \cap D = \emptyset $.
Therefore, $D \in R_{\alpha}$, as needed.

\emph{Case 2:} The final case is when $N_{j_0} \cap D_{\beta} \cap D = \emptyset $.
In this case, since $\rows(D_{\alpha}) \subset \rows(D_{\beta})$,
we have that $N_{j_0} \cap D_{\alpha} \cap D = \emptyset $ as well.
So $D \in R_{\alpha}$.
So we have shown that $\overline{R}_{\beta} \subset R_{\alpha}$, as needed.
\end{proof}

\begin{rmk}\label{rmk:MaxCliquesFactorOut}
Note that Proposition \ref{prop:MaxCliquesFactorOut} implies that whenever $D_{\alpha}$ and $D_{\beta}$
are covered by the same element of $P(j_0)$,
we have that $R_{\alpha} = R_{\beta}$.
This shows that the left-hand side of Equation (\ref{eqn:SumInBlock}) for $D_{\alpha}$ and $D_{\beta}$ consist of the same terms that come from cliques in $\Max(S)$.
\end{rmk}

Let $D_{\alpha_1}, \dots, D_{\alpha_{\ell}} \lessdot D_{\beta}$. As we discussed in Remark \ref{rmk:InductionIdea}, in order to sum the values of $x_{ij_0}$ over $N_{j_0} \cap D_{\beta}$, we must understand the sum over $x_{ij_0}$ for those rows $i$ such that $i \in \rows(D_{\beta})$ but $i \not\in \rows(D_{\alpha_k})$ for all $k$. The following proposition concerns the sum of the $x_{ij_0}$ over these values of $i$. 

\begin{prop}\label{prop:SumExtraEntries}
Let $D_{\alpha_1}, \dots, D_{\alpha_{\ell}} \lessdot D_{\beta}$. Let $r_1, \dots, r_a$ be the rows of $D_{\beta}$ that are not in any $D_{\alpha_k}$ for $k=1, \dots, \ell$. Then
\begin{equation}\label{eqn:SumExtraEntries}
\sum_{i=1}^a  x_{r_i j_0} = 
 u_{+j_0} \Big( \prod_{\substack{C \in \Int(S) \\ C \cap N_{j_0} \neq \emptyset \\ \rows(D_{\beta}) \subset \rows(C)}}C^+ \Big)
\Big(\prod_{\substack{D \in \Max(S) \\  D \cap N_{j_0} \neq \emptyset \\ \rows(D) \subsetneq \rows(D_{\beta})}} D^+ \Big)
\Big(\prod_{\substack{ E \in \Max(S) \\ N_{j_0} \cap D_{\beta} \cap E = \emptyset}} E^+\Big)
\Big(\sum_{i=1}^a u_{r_i +}\Big)
\end{equation}
\end{prop}

%changed r to a throughout this because we had already used r, so if you see an r that doesn't make sense, that's why and it should be changed

\begin{proof}
Without loss of generality, we will let $D_2, \dots, D_{\ell} \lessdot D_{1}$, 
and let rows $1, \dots, a$ be the rows of $D_{1}$ that are not rows of any $D_\alpha$ for $\alpha = 2,\dots,\ell$. 
Let $i  \in [a]$.
Recall that
\[
x_{ij_0} = u_{i+}u_{+j_0}\displaystyle{\prod_{C \in \Int(ij_0)} C^+}\displaystyle{ \prod_{D \in \Max(S) \setminus \Max(ij_0)}D^+}.
\] 
We first consider the maximum cliques $D$ with $(i,j_0) \not\in D$. 
If $N_{j_0} \cap D_{1} \cap D = \emptyset$, then $D^+$ is a term of $x_{ij_0}$ for all $i = 1,\dots,a$.
So $D^+$ is a factor of both the left-hand and right-hand sides of Equation (\ref{eqn:SumExtraEntries}).

Otherwise, we have $N_{j_0} \cap D_1 \cap D \neq \emptyset$.
In particular, this means that $N_{j_0} \cap D \neq \emptyset$.
Since $\rows(D_1) \cap \rows(D) \neq \emptyset$, and $(i,j_0) \not\in D$,
we must have $\rows(D) \subsetneq \rows(D_1)$ by the DS-free condition.
Furthermore, for all maximal cliques $D$ with $\rows(D) \subsetneq \rows(D_1)$,
 we have $(i,j_0) \not\in D$.
Indeed, $D \cap N_{j_0} \neq \emptyset$, so by Proposition \ref{prop:MaxCliqueCharacterization}, 
we have $D = D_{\gamma}$ for some $\gamma$
with $D_{\gamma} < D_1$ in $P(j_0)$.
So $\rows(D_{\gamma}) \subset \rows(D_{\alpha})$ for some $\alpha \in \{2,\dots,\ell\}$.
Since $(i,j_0) \not\in D_{\alpha}$, we have $(i,j_0) \not\in D_{\gamma}$ as well.

Therefore, the factors $D^+$ corresponding to maximal cliques in each $x_{ij_0}$ are the same for all $i \in [a]$,
and are exactly those with $\rows(D) \subsetneq \rows(D_1)$ or $N_{j_0} \cap D_1 \cap D = \emptyset$.

Now let $(i,j_0) \in C$ where $C \in \Int(S)$ and $C \cap N_{j_0} \neq \emptyset$. 
By Proposition \ref{prop:MaxIntCharacterization}, we have $C = D_{\gamma} \cap D_{\delta}$ 
where $D_{\gamma} \lessdot D_{\delta}$ in $P(j_0)$. 
Since $C \cap D_1$ is nonempty, we must have that
$\rows(D_{\gamma}) \subset \rows(D_1)$ or $\rows(D_1) \subset \rows(D_{\gamma})$,
and similarly for $D_{\delta}$.
But since $i \in [a]$, $(i,j_0) \not\in D_{\alpha}$ for any $D_{\alpha} < D_1$.
So we must have $\rows(D_1) \subset \rows(D_{\gamma}) \subset \rows(D_{\delta}).$
Therefore, $\rows(D_1) \subset \rows(C)$.

Furthermore, we have that $(i,j_0) \in C$ for all $C \in \Int(S)$ with $\rows(D_1) \subset \rows(C)$.
So the factors $C^+$ corresponding to maximal intersections of maximal cliques are exactly
those with $\rows(D_1) \subset \rows(C)$ in each $x_{ij_0}$.

Therefore, we have that
\begin{align*}
\sum_{i=1}^a x_{ij_0} & =  \sum_{i=1}^a u_{i+}u_{+j_0}\Big(\prod_{C \in \Int(ij_0)} C^+ \Big) \Big(\prod_{D \in \Max(S) \setminus \Max(ij_0)}D^+\Big) \\
&= \Big( \prod_{\substack{C \in \Int(S)  \\ C \cap N_{j_0} \neq \emptyset \\ \rows(D_1) \subset \rows(C)}}C^+ \Big)
\Big(\prod_{\substack{D \in \Max(S) \\  D \cap N_{j_0} \neq \emptyset \\ \rows(D) \subsetneq \rows(D_1)}} D^+ \Big)
\Big(\prod_{\substack{ E \in \Max(S) \\ N_{j_0} \cap D_1 \cap E = \emptyset}} E^+\Big)
\Big(\sum_{i=1}^a u_{i +} u_{+j_0}\Big) \\
%&= \Big( \prod_{\substack{C \in \Int(S) \\ \rows(D_1) \subset \rows(C)}}C^+ \Big)
%\Big(\prod_{\substack{D \in \Max(S) \\ \rows(D) \subsetneq \rows(D_1)}} D^+ \Big)
%\Big(\prod_{\substack{ E \in \Max(S) \\ N_{j_0} \cap D_1 \cap E = \emptyset}} E^+\Big)
%\Big(  u_{+j_0} \Big) \Big(\sum_{i=1}^r u_{i+} \Big) \\
&= u_{+j_0}\Big( \prod_{\substack{C \in \Int(S)  \\ C \cap N_{j_0} \neq \emptyset \\ \rows(D_1) \subset \rows(C)}}C^+ \Big)
\Big(\prod_{\substack{D \in \Max(S) \\  D \cap N_{j_0} \neq \emptyset \\ \rows(D) \subsetneq \rows(D_1)}} D^+ \Big)
\Big(\prod_{\substack{ E \in \Max(S) \\ N_{j_0} \cap D_1 \cap E = \emptyset}} E^+\Big)
\Big(\sum_{i=1}^a u_{i+} \Big),
\end{align*}
as needed.
\end{proof}

Finally, the following proposition gives a way to write $D_{\beta}^+$ as a sum over its intersections with the elements of $P(j_0)$ that it covers, along with the rows of $D_{\beta}$ that are not rows of any clique that it covers.

\begin{prop}\label{prop:SumOfUnfactoredTerms}
Let $D_{\alpha_1}, \dots, D_{\alpha_{\ell}} \lessdot D_{\beta}$. Let $r_1, \dots, r_a$ be the rows of $D_{\beta}$ that are not in any $D_{\alpha_i}$ for $i=1, \dots, \ell$. Then
\begin{equation}\label{eqn:SumOfUnfactoredTerms}
D_{\beta}^+ = \sum_{i=1}^a u_{r_i +} + \sum_{i=1}^{\ell} (D_{\alpha_i} \cap D_{\beta})^+.
\end{equation}
\end{prop}

\begin{proof}
Without loss of generality, we will let $D_2, \dots, D_{\ell} \lessdot D_{1}$, 
and let rows $1, \dots, a$ be the rows of $D_1$ that are not rows of any $D_\alpha$ for $\alpha = 2, \dots,\ell$. 
First note that each $u_{ij}$ that appears on the right-hand side of Equation (\ref{eqn:SumOfUnfactoredTerms})
is a term of $D_1^+$.
Indeed, if $(i,j) \in D_{\alpha} \cap D_1$ for some $\alpha = 2,\dots,\ell$, this is clear.

Otherwise, we have $i \in [a]$.
For the sake of contradiction, suppose that there exists a column $j$ so that $(i,j) \not\in D_1$ but $(i,j) \in S$.
But then $\rows(D_1) \cap \rows(j)$ is non-empty.
So by the DS-free condition, either $\rows(D_1) \subset \rows(j)$ or $\rows(j) \subsetneq \rows(D_1)$.
If $\rows(D_1) \subset \rows(j)$, then $j$ is a column of $D_1$ by definition, which is a contradiction.
If $\rows(j) \subsetneq \rows(D_1)$, then column $j$ belongs to some block $B_{\eta}$
with $\rows(D_{\eta}) \subsetneq \rows(D_1)$.
But this contradicts that row $i$ is not in any $D_{\alpha}$ for $\alpha = 2,\dots,\ell$.

Now it remains to show that all the terms in $D_1^+$ appear in the right-hand side of Equation (\ref{eqn:SumOfUnfactoredTerms}).
Let $(i,j) \in D_1$.
If $i \in [a]$, then $u_{ij}$ is a term in the right-hand side, as needed.
Otherwise, $i \in \rows(D_{\alpha})$ for some $\alpha \in \{2,\dots,\ell\}$.
Since $\cols(D_1) \subset \cols(D_{\alpha})$ by definition,
we must have $j \in \cols(D_{\alpha})$. So $(i,j) \in D_{\alpha}$.
Therefore, $u_{ij}$ is a term in $(D_{\alpha} \cap D_1)^+$.
Finally, since $D_{\gamma} \cap D_{\delta} = \emptyset$ for all
 $\gamma, \delta \in \{2,\dots,l\}$ with $\gamma \neq \delta$,
no term is repeated.
\end{proof}

%%%%%%%%%%%%%%%%%%%%%%%%%%%%%%%%%%%%%%%%%%%%%%%%%%%%%%%%%%%
%%%%%%%%%%%%%%%%%%%%%%%%%%%%%%%%%%%%%%%%%%%%%%%%%%%%%%%%%%%
%%%%%%%%%%%%%%%%%%%%%%%%%%%%%%%%%%%%%%%%%%%%%%%%%%%%%%%%%%%
%%%%%%%%%%%%%%%%%%%%%%%%%%%%%%%%%%%%%%%%%%%%%%%%%%%%%%%%%%%

We can now use these propositions to prove Lemma \ref{lem:SumInBlock}.

\begin{proof}[Proof of Lemma \ref{lem:SumInBlock}]
We will induct over the poset $P(j_0)$. 
For the base case, we let $D_{\alpha}$ be minimal in $P(j_0)$. 
First, by Proposition \ref{prop:SumOfUnfactoredTerms} we have that
\[
D_{\alpha}^+ = \sum_{i \in \rows(D_{\alpha})} u_{i+}
\]
since $D_{\alpha}$ does not cover any element of $P(j_0)$.
%Indeed, if this were not the case, then by Proposition \ref{prop:SumOfUnfactoredTerms},
%there would be a column $j$ such that for two rows $i_1, i_2 \in \rows(D_{\alpha})$, 
%we have $i_1 \in \rows(j)$ and $i_2 \not\in \rows(j)$. 
%Since $\rows(D_{\alpha}) \cap \rows(j) \neq \emptyset$, 
%and $\rows(D_{\alpha}) \not\subset \rows(j)$, 
%this implies that $\rows(j) \subsetneq \rows(D_{\alpha})$ by the DS-free condition. 
%But this contradicts that $D_{\alpha}$ is minimal in $P(j_0)$.

%It is clear that for any $C \in \Int(S)$ with $\rows(D_{\alpha}) \subset \rows(C)$
%and any $i$ such that $(i,j_0) \in D_{\alpha}$, we have $(i,j_0) \in C$.
%So $C^+$ is a factor of $x_{ij_0}$.

Let $C \in \Int(S)$ with $\rows(D_{\alpha}) \subset \rows(C)$
such that $C \cap N_{j_0} \neq \emptyset$.
%By Proposition \ref{prop:IntersectionsFactorOut},
%$C$ is of the form $D_{\beta} \cap D_{\gamma}$
Then for any $i \in \rows(D_{\alpha})$, $(i,j_0) \in C$.
So $C \in \Int(ij_0)$ and $C^+$ is a factor of $x_{ij_0}$.

If $C \in \Int(ij_0)$, then $N_{j_0} \cap C \neq \emptyset$.
It remains to be shown that all factors of $x_{ij_0}$, $C^+$ corresponding to maximal intersections of maximal cliques
have $\rows(D_{\alpha}) \subset \rows(C)$. 

Let $(i,j_0) \in D_{\alpha}$. 
Let $D_{\beta}$ and $D_{\gamma}$ be maximal cliques such that $C = D_{\beta} \cap D_{\gamma} \in \Int(ij_0)$.
Then we have $\rows(D_{\alpha}) \cap \rows(D_{\beta})$ and $\rows(D_{\alpha}) \cap \rows(D_{\gamma})$ nonempty.
Since $D_{\alpha}$ is minimal, this implies that $\rows(D_{\alpha}) \subset \rows(D_{\beta}), \rows(D_{\gamma})$.
So $\rows(D_{\alpha}) \subset \rows(C)$.
Therefore the factors of each $x_{ij_0}$ with $(i,j_0) \in D_{\alpha}$
that correspond to maximal intersections of maximal cliques are
\[
\prod_{\substack{C \in \Int(S) \\ C \cap N_{j_0} \neq \emptyset \\ \rows(D_{\alpha}) \subset \rows(C)}} C^+,
\]
as needed. The other factors of each $x_{ij_0}$ for $(i,j_0) \in D_{\alpha}$ are of the form
\[
\prod_{\substack{E \in \Max(S) \\ (i,j_0) \not\in E}} E^+.
\]
Since all $(i,j_0) \in D_{\alpha}$ are contained in the same maximal cliques when $D_{\alpha}$ is minimal in $P(j_0)$, 
the terms corresponding to maximal cliques in each $x_{ij_0}$ are of the form
\[
\prod_{\substack{E \in \Max(S) \\ N_{j_0} \cap D_{\alpha} \cap E= \emptyset}} E^+,
\]
as needed.
So we have that
\begin{align*}
\sum_{(i,j_0) \in  D_{\alpha}} x_{ij_0} &= \sum_{(i,j_0) \in D_{\alpha}} u_{i+} u_{+j_0}
\Big(\prod_{\substack{C \in \Int(S) \\ C \cap N_{j_0} \neq \emptyset \\ \rows(D_{\alpha}) \subset \rows(C)}} C^+ \Big)
\Big( \prod_{\substack{E \in \Max(S) \\ N_{j_0} \cap D_{\alpha} \cap E= \emptyset}} E^+ \Big) \\
&= u_{+j_0} \Big( \sum_{i \in \rows(D_{\alpha})} u_{i+} \Big)
\Big(\prod_{\substack{C \in \Int(S)  \\ C \cap N_{j_0} \neq \emptyset \\ \rows(D_{\alpha}) \subset \rows(C)}} C^+ \Big)
\Big( \prod_{\substack{E \in \Max(S) \\ N_{j_0} \cap D_{\alpha} \cap E= \emptyset}} E^+ \Big) \\
&= u_{+j_0} D_{\alpha}^+ \Big(\prod_{\substack{C \in \Int(S)  \\ C \cap N_{j_0} \neq \emptyset\\ \rows(D_{\alpha}) \subset \rows(C)}} C^+ \Big)
\Big( \prod_{\substack{E \in \Max(S) \\ N_{j_0} \cap D_{\alpha} \cap E= \emptyset}} E^+ \Big).
\end{align*}

Since $D_{\alpha}$ is the only maximal clique whose rows are contained in $D_{\alpha}$, we have  that
\[
D_{\alpha}^+ = \prod_{\substack{D \in \Max(S) \\ D \cap N_{j_0} \neq \emptyset \\ \rows(D) \subset \rows(D_{\alpha})}} D^+. 
\]
So the lemma holds for the base case.

Without loss of generality, let $D_2, \dots, D_{\ell} \lessdot D_1$ in $P(j_0)$. Let rows $1, \dots, a$ be the rows of $D_1$ that are not in any $D_{\alpha}$ with $\alpha = 2, \dots,\ell$. We have the following chain of equalities.

\begin{align*} \allowdisplaybreaks[4]
\sum_{(i,j_0) \in D_1} x_{ij_0} &= u_{+j_0} \sum_{i=1}^a  x_{ij_0} + \sum_{\alpha=2}^{\ell}  \sum_{(i,j_0) \in D_{\alpha}} x_{ij_0} \\
&=  u_{+j_0} \Big( \prod_{\substack{C \in \Int(S)  \\ C \cap N_{j_0} \neq \emptyset \\ \rows(D_{\beta}) \subset \rows(C)}}C^+ \Big)
\Big(\prod_{\substack{D \in \Max(S) \\  D \cap N_{j_0} \neq \emptyset \\ \rows(D) \subsetneq \rows(D_{\beta})}} D^+ \Big)
\Big(\prod_{\substack{ E \in \Max(S) \\ N_{j_0} \cap D_{\beta} \cap E = \emptyset}} E^+\Big)
\Big(\sum_{i=1}^p u_{i+} \Big) \\
& \qquad \qquad + \sum_{\alpha=2}^{\ell}  \sum_{(i,j_0) \in D_{\alpha}} x_{ij_0}\\
&= u_{+j_0}\Big( \prod_{\substack{C \in \Int(S)  \\ C \cap N_{j_0} \neq \emptyset \\ \rows(D_{\beta}) \subset \rows(C)}}C^+ \Big)
\Big(\prod_{\substack{D \in \Max(S) \\ D \cap N_{j_0} \neq \emptyset \\ \rows(D) \subsetneq \rows(D_{\beta})}} D^+ \Big)
\Big(\prod_{\substack{ E \in \Max(S) \\ N_{j_0} \cap D_{\beta} \cap E = \emptyset}} E^+\Big)
\Big(\sum_{i=1}^a u_{i+} \Big) \\
& \qquad \qquad + \sum_{\alpha =2}^{\ell} u_{+j_0} 
\Big( \prod_{\substack{C \in \Int(S)  \\ C \cap N_{j_0} \neq \emptyset \\ \rows(D_{\alpha}) \subset \rows(C)}} C^+ \Big)
\Big( \prod_{\substack{D \in \Max(S) \\ D \cap N_{j_0} \neq \emptyset \\ \rows(D) \subset \rows(D_{\alpha})}} D^+ \Big)
\Big( \prod_{\substack{E \in \Max(S) \\ N_{j_0} \cap D_{\alpha} \cap E = \emptyset}} E^+ \Big) \\
&= u_{+j_0} \Big(\prod_{\substack{D \in \Max(S) \\ D \cap N_{j_0} \neq \emptyset \\ \rows(D) \subsetneq \rows(D_1)}} D^+ \Big)
\Big( \prod_{\substack{E \in \Max(S) \\ N_{j_0} \cap D_1 \cap E = \emptyset}} E^+ \Big) \\
& \qquad \qquad \times
\Big(\Big( \sum_{i=1}^a u_{i+} \times \prod_{\substack{C \in \Int(S)  \\ C \cap N_{j_0} \neq \emptyset \\ \rows(D_1) \subset \rows(C)}} C^+ \Big)
+ \Big(\sum_{\alpha =2}^{\ell} \prod_{\substack{C \in \Int(S) \\ \rows(D_{\alpha}) \subset \rows(C)}} C^+ \Big) \Big) \\
&= u_{+j_0} \Big(\prod_{\substack{D \in \Max(S) \\ D \cap N_{j_0} \neq \emptyset \\ \rows(D) \subsetneq \rows(D_1)}} D^+ \Big)
\Big( \prod_{\substack{E \in \Max(S) \\ N_{j_0} \cap D_1 \cap E = \emptyset}} E^+ \Big) \\ 
& \qquad \qquad \times
\Big( \prod_{\substack{C \in \Int(S)  \\ C \cap N_{j_0} \neq \emptyset \\ \rows(D_1) \subset \rows(C)}} C^+ \Big)
\Big( \sum_{i=1}^a u_{i+} + \sum_{\alpha =2}^{\ell} (D_{\alpha} \cap D_1)^+ \Big)  \\
&=u_{+j_0} \Big(\prod_{\substack{D \in \Max(S \\ D \cap N_{j_0} \neq \emptyset) \\ \rows(D) \subsetneq \rows(D_1)}} D^+ \Big)
\Big( \prod_{\substack{E \in \Max(S) \\ N_{j_0} \cap D_1 \cap E = \emptyset}} E^+ \Big)
\Big( \prod_{\substack{C \in \Int(S)  \\ C \cap N_{j_0} \neq \emptyset \\ \rows(D_1) \subset \rows(C)}} C^+ \Big)
\Big( D_1^+ \Big) \\ 
&= u_{+j_0} \Big(\prod_{\substack{D \in \Max(S) \\ D \cap N_{j_0} \neq \emptyset \\ \rows(D) \subset \rows(D_1)}} D^+ \Big)
\Big( \prod_{\substack{E \in \Max(S) \\ N_{j_0} \cap D_1 \cap E = \emptyset}} E^+ \Big)
\Big( \prod_{\substack{C \in \Int(S)  \\ C \cap N_{j_0} \neq \emptyset \\ \rows(D_1) \subset \rows(C)}} C^+ \Big)
\end{align*}

The second equality follows from Proposition \ref{prop:SumExtraEntries}. The third equality is an application of the inductive hypothesis. The fourth equality follows from Proposition \ref{prop:MaxCliquesFactorOut} along with Remark \ref{rmk:MaxCliquesFactorOut}. The fifth equality follows from Proposition \ref{prop:IntersectionsFactorOut}. The sixth equality follows from Proposition \ref{prop:SumOfUnfactoredTerms}. The seventh inequality follows from the fact that $D_1$ is the only clique whose rows are equal to $\rows(D_1)$. This completes our proof by induction.
\end{proof}

\section{Checking the Conditions of Birch's Theorem}\label{sec:BirchsThm}

In the previous section, we wrote a formula for the sum of $x_{ij_0}$ where
$i$ ranges over the rows of some maximal clique $D_{\alpha}$.
Since the block $B_0$ induces its own maximal clique,
Lemma \ref{lem:SumInBlock} allows us to write the sum of the $x_{ij_0}$s for $1 \leq i \leq r$ in the following concise way.
This in turn verifies that the proposed maximum likelihood estimate $\hat{p}$
has the same sufficient statistics as the normalized data $u / u_{++}$,
which is one of the conditions of Birch's theorem.

\begin{cor}\label{cor:SumEntireBlock}
Let $S$ be DS-free. Then for any column $j_0$,
\[
\sum_{i=1}^r x_{ij_0} = u_{+j_0} \prod_{D \in \Max(S)} D^+.
\]
\end{cor}

\begin{proof}
The poset $P(j_0)$ has a unique maximal element $D_0$ with $\rows(D_0) = \rows(j_0)$.
Note that $D_0$ may include more columns than $j_0$ 
since it may have columns whose nonzero rows are the same as or contain those of $j_0$.

By Proposition \ref{prop:MaxIntCharacterization}, there are no maximal intersections of maximal cliques $C$
with $\rows(D_0) \subset \rows(C)$, since $D_0$ is maximal in $P(j_0)$.
It follows from  Proposition \ref{prop:MaxCliqueCharacterization}
that a maximal clique $D$  intersects $N_{j_0}$ if and only if it has $\rows(D) \subset \rows(D_0)$.

Since $N_{j_0} \subset D_0$, we have that $N_{j_0} \cap D_0 \cap E = N_{j_0} \cap E$ for any clique $E$.
By Lemma \ref{lem:SumInBlock}, we have
\begin{align*}
\sum_{i=1}^r x_{ij_0} &= u_{+j_0}  \Big( \prod_{\substack{D \in \Max(S) \\ D \cap N_{j_0} \neq \emptyset \\ \rows(D) \subset \rows(D_0)}}D^+ \Big) 
\Big( \prod_{\substack{E \in \Max(S) \\ N_{j_0} \cap D_0 \cap E = \emptyset}} E^+ \Big)\\ 
&=  u_{+j_0} \Big( \prod_{\substack{D \in \Max(S) \\ D \cap N_{j_0} \neq \emptyset}} D^+ \Big)
 \Big(\prod_{\substack{E \in \Max(S) \\ E \cap N_{j_0} = \emptyset}} E^+ \Big) \\
 &=  u_{+j_0} \prod_{D \in \Max(S)} D^+,
\end{align*}
as needed.
\end{proof}

Now we will address the condition of Birch's theorem which states that the maximum likelihood estimate must satisfy the equations defining $\Mcal_S$.

\begin{lemma}\label{lem:MLEinModel}
Let $S$ be doubly chordal bipartite. Let $u \in \R^{S}$ be a generic matrix of counts. Then the point $(\hat{p}_{ij} \mid (i,j) \in S)$ specified in Theorem \ref{thm:Main} is in the Zariski closure of $\Mcal_S$.
\end{lemma}

In order to prove this lemma, we must first describe the vanishing ideal of $\Mcal_S$. We denote this ideal $\Ical(\Mcal_S)$. It is a subset of the polynomial ring in $\#S$ variables, 
\[R = \mathbb{C}[p_{ij} \mid ij \in S].\]

\begin{prop}\label{prop:DefiningEquations}
Let $S$ be chordal bipartite. Then $\Ical(\Mcal_S)$ is generated by the $2 \times 2$ minors of the matrix form of $S$ that contain no zeros. That is, $\Ical(\Mcal_S)$ is generated by all binomials of the form
\[
p_{ij} p_{k\ell} - p_{i \ell} p_{kj},
\]
such that $(i,j), (k,\ell), (i,\ell), (k,j) \in S$.
\end{prop}

\begin{proof}
This follows from results in \cite[Chapter~10.1]{aoki2012}. The \emph{loops} on $S$ correspond to cycles in $G_S$. The  \emph{df 1 loops} as defined in \cite[Chapter~10.1]{aoki2012} are those whose support does not properly contain the support of any other loop; that is, they correspond to cycles in $G_S$ with no chords. Since $G_S$ is chordal bipartite, each of these cycles contain exactly four edges. Therefore the df 1 loops on $S$ all have degree two, and each corresponds to a $2 \times 2$ minor of $S$ by definition. Theorem 10.1 of \cite{aoki2012} states that the df 1 loops form a Markov basis for $\Mcal_S$. Therefore, by the Fundamental Theorem of Markov Bases \cite[Theorem~3.1]{diaconis1998}, the $2 \times 2$ minors of $S$ form a generating set for $\Ical(\Mcal_S)$.
\end{proof}

\begin{ex}
	Consider the matrix $S$ from Example \ref{ex:RunningExample}.
	In Figure \ref{Fig:RunningGraph}, we see that $G(S)$ has exactly one cycle. This cycle corresponds to the only $2\times2$ minor in $S$ that contains no zeros,
	which is the $\{2,3\} \times \{1,2\}$ submatrix.
	Therefore the (complex) Zariski closure of $\Mcal_S$ is the variety of the ideal generated by the polynomial $p_{21}p_{32} - p_{31}p_{22}$.
\end{ex}

\begin{prop}\label{prop:ModelEquationsSatisfied}
Let $S$ be set of indices such that $G_S$ is doubly chordal bipartite. Let $\{i_1,i_2\} \times \{j_1,j_2\}$ be a set of indices that corresponds to a $2 \times 2$ minor of $S$ that contains no zeros. Let $\hat{p}_{i_1j_1}, \hat{p}_{i_2 j_2}, \hat{p}_{i_1j_2}, \hat{p}_{i_1j_2}$ be as defined in Theorem \ref{thm:Main}. Then 
\begin{equation}\label{eqn:2by2minor}
\hat{p}_{i_1j_1}\hat{p}_{i_2j_2} = \hat{p}_{i_1j_2} \hat{p}_{i_2 j_1}
\end{equation}

\end{prop}

\begin{proof}
The terms $u_{i_1+}, u_{i_2+}, u_{+j_1}$ and $u_{+j_2}$ each appear once in the numerator on each side of Equation (\ref{eqn:2by2minor}), and $u_{++}^2$ appears in both denominators. Furthermore if $(i_1,j_1)$ and $(i_2,j_2)$ are both contained in any clique in $S$, then $(i_1,j_2)$ and $(i_2,j_1)$ are also in the clique by definition. So any term that is squared in the numerator or denominator on one side of Equation (\ref{eqn:2by2minor}) is also squared on the other side. Therefore it suffices to show that $\Max(i_1j_1) \cup \Max(i_2j_2) = \Max(i_1j_2) \cup \Max(i_2j_1)$ and $\Int(i_1j_1) \cup \Int(i_2j_2) = \Int(i_1j_2) \cup \Int(i_2j_1)$.

First, we will show that $\Max(i_1j_1) \cup \Max(i_2j_2) = \Max(i_1j_2) \cup \Max(i_2j_1)$. 
Let $D \in \Max(i_1j_1)$. If $(i_2,j_1) \in D$, then we are done. 

Now suppose that $(i_2,j_1) \not\in D$. 
Since $D$ intersects column $j_1$, by Proposition \ref{prop:MaxCliqueCharacterization} we know that $D$ has the form $D_{\alpha}$ 
for some block of columns $B_{\alpha}$ that are identical on $\rowsone(j_1)$. 
Let $\rowsone(D_{\alpha})$ denote the set of nonzero rows of $D_{\alpha}$ that are also nonzero rows of $j_1$. 
Since $(i_2,j_1) \not\in D$, we have that $i_2 \not\in \rowsone(D_{\alpha})$ while $i_1 \in \rowsone(D_{\alpha})$.
Since $\rowsone(j_2) \cap \rowsone(D_{\alpha})$ is nonempty, and since $\rowsone(j_2) \not\subset \rowsone(D_{\alpha})$,
we must have that $\rowsone(D_{\alpha}) \subset \rowsone(j_2)$ by the DS-free condition.
Therefore $(i_1,j_2) \in D_{\alpha}$ by definition of $D_{\alpha}$.
So $D_{\alpha} = D \in \Max(i_1,j_2)$, as needed.

Switching the roles of $i_1$ and $i_2$ or the roles of $j_1$ and $j_2$ yields the desired equality.

Now let $C \in \Int(i_1,j_1)$. Then $C = D_{\alpha} \cap D_{\beta}$ where $D_{\beta} \lessdot D_{\alpha}$ in the poset $P(j_1)$ by Proposition \ref{prop:MaxIntCharacterization}.
If $(i_2, j_1) \in C$, then we are done.

Now suppose that $(i_2,j_1) \not\in C$. 
Then we have that $i_2 \not\in \rowsone(D_{\beta})$, whereas $i_1 \in \rowsone(D_{\alpha})$ and $i_1 \in \rowsone(D_{\beta})$. 
So we must have that $\rowsone(D_{\beta}) \subsetneq \rowsone(j_2)$ by the DS-free condition.
Since $\rowsone(D_{\alpha}) \cap \rowsone(j_2)$ is nonempty, we must have that
$\rowsone(D_{\alpha}) \subset \rowsone(j_2)$. This follows from the DS-free condition and the fact that $D_{\alpha}$ covers $D_{\beta}$ in the poset $P(j_1)$.
Therefore $(i_1,j_2) \in D_{\alpha}, D_{\beta}$ by definition of these cliques.
So $C \in \Int(i_1,j_2)$, as needed.

Again, switching the roles of $i_1$ and $i_2$ or the roles of $j_1$ and $j_2$ in the above proof yields the desired equality.
\end{proof}
%
%\begin{ex}\label{ex:ModelEquationsSatisfied}
%Consider the following matrix with structural zeros in which the $\{1,2\} \times \{1,2\}$ minor is an intact submatrix:
%\begin{equation}\label{eqn:ModelEquationsSatisfied1}
%S_1 = \begin{bmatrix}
%\star & \star & 0 \\
%\star & \star & \star \\
%\star & \star & \star \\
%0 & \star & \star
%\end{bmatrix}
%\end{equation}
%The clique $D = \{2,3,4\}\times \{2,3\}$ intersects the $\{1,2\} \times \{1,2\}$ submatrix only at $(2,2)$, which must be avoided in order for the model equations of $\Mcal_S$ to be satisfied.
%This happens because $S_1$ is not chordal bipartite. Indeed, the $\{1,2,4\} \times \{1,2,3\}$ submatrix is the matrix of a double-square.
%\end{ex}

\begin{proof}[Proof of Lemma \ref{lem:MLEinModel}]
By Proposition \ref{prop:DefiningEquations}, the vanishing ideal of $\Mcal_S$ consists of all fully-observed $2 \times 2$ minors of $S$.
By Proposition \ref{prop:ModelEquationsSatisfied}, each of these $2 \times 2$ minors vanishes when evaluated on $\hat{p}$.
\end{proof}

We can now prove Theorem \ref{thm:Main}.

\begin{proof}[Proof of Theorem \ref{thm:Main}]
Let $G_S$ be doubly chordal bipartite.
Let $u \in \R_+^S$ be a matrix of counts.
By Corollary \ref{cor:SumEntireBlock}, the column marginals of $u_{++} \hat{p}$ are equal to those of $u$. 
Switching the roles of rows and columns in all of the proofs used to obtain this corollary shows that the row marginals are also equal. 
Corollary \ref{cor:SumEntireBlock} also implies that $\hat{p}_{++} = 1$ since the vector of all ones is in the rowspan of $A(S)$.
So by Lemma \ref{lem:MLEinModel} and the fact each $\hat{p}$ is positive, $\hat{p} \in \Mcal_S$. 
Hence by Birch's theorem, $\hat{p}$ is the maximum likelihood estimate for $u$.
The other direction is exactly the contrapositive of Theorem \ref{thm:NotDCB}.
\end{proof}

\section*{Acknowledgments}

Jane Coons was partially supported by the US National Science Foundation (DGE 1746939).
 Seth Sullivant was partially supported by the US National Science Foundation (DMS 1615660).

\bibliographystyle{acm}
% Include the ".bib" file (generated by bibtex) right here.

\bibliography{quasi-indep.bib}

\end{document}